\documentclass{amsart}
\usepackage{amsmath}
\usepackage{amssymb}
\usepackage{amsthm}
\usepackage[margin=1.375in]{geometry}
\usepackage{amscd}
\usepackage{enumerate}
\usepackage{hyperref}
\usepackage[nameinlink]{cleveref}

\usepackage{graphicx}
\usepackage{amsmath,amsthm,amsfonts,amscd,amssymb,comment,eucal,latexsym,mathrsfs}
\usepackage{stmaryrd}
\usepackage[all]{xy}

\usepackage{epsfig}

\usepackage[all]{xy}
\xyoption{poly}
\usepackage{fancyhdr}
\usepackage{wrapfig}
\usepackage{epsfig}

\usepackage{hyperref}
\allowdisplaybreaks

\theoremstyle{plain}
\newtheorem{thm}{Theorem}[section]
\newtheorem{prop}[thm]{Proposition}
\newtheorem{lem}[thm]{Lemma}
\newtheorem{cor}[thm]{Corollary}
\newtheorem{conj}{Conjecture}

\theoremstyle{definition}
\newtheorem{defn}[thm]{Definition}
\newtheorem{rmk}[thm]{Remark}
\theoremstyle{remark}

 \def\C{{\mathbb{C}}}   \def\E{{\mathbb{E}}}\def\N{{\mathbb{N}}} \def\R{{\mathbb{R}}}  

\newcommand\vp{\varphi}

\newcommand\Aut{\operatorname{Aut}}

\newcommand\id{\operatorname{id}}

\newcommand{\LRig}{\operatorname{L^2Rig}}

\newcommand{\MaxR}{\operatorname{MaxRig}}

\renewcommand\Re{\operatorname{Re}}

\newcommand\Tr{\operatorname{Tr}}

\newcommand\Span{\operatorname{span}}

\def\cc{{\curvearrowright}}
\newcommand{\actson}{\curvearrowright}

\newcommand{\ip}[1]{\langle #1 \rangle}
\newcommand{\cal}{\mathcal}

\begin{document}

\title{Maximal rigid subalgebras of deformations and $L^{2}$-cohomology}
\date{\today}

\author[R. de Santiago]{Rolando de Santiago}
\address{Department of Mathematics, University of California Los Angeles, Box 951555, 520 Portola Plaza, Los Angeles, CA 90095-1555, USA.}
\email{rdesantiago@math.ucla.edu}
\thanks{R.dS. was partially supported by NSF RTG DMS-1344970 and by the University of California President’s Postdoctoral Fellowship.}

\author[B. Hayes]{Ben Hayes}
\address{Department of Mathematics,
University of Virginia, 141 Cabell Drive, Kerchof Hall, Charlottesville, VA, 22904, USA.}
\email{brh5c@virginia.edu}
\thanks{B.H. was partially supported by NSF Grants DMS-1600802 and DMS-1827376.}

\author[D.J. Hoff]{Daniel J. Hoff}
\address{Department of Mathematics, University of California Los Angeles, Box 951555, 520 Portola Plaza, Los Angeles, CA 90095-1555, USA.}
\email{hoff@math.ucla.edu}
\thanks{D.H. was partially supported by NSF RTG DMS-1344970.}

\author[T. Sinclair]{Thomas Sinclair}
\address{Department of Mathematics, Purdue University, 150 N.~University St., West Lafayette, IN 47907-2067, USA.}
\email{tsincla@purdue.edu}
\thanks{T.S. was partially supported by NSF Grant DMS-1600857}

\subjclass[2010]{46L36; 46L10, 37A55}

\vspace*{-.8cm}
\maketitle

\vspace{-.5cm}
\begin{abstract}
In the past two decades, Sorin Popa's breakthrough deformation/rigidity theory has produced remarkable rigidity results for von Neumann algebras $M$ which can be deformed inside a larger algebra $\widetilde M \supseteq M$ by an action $\alpha: \mathbb{R} \to {\rm Aut}(\widetilde M)$, while simultaneously containing subalgebras $Q$ {\it rigid} with respect to that deformation, that is, such that $\alpha_t \to {\rm id}$ uniformly on the unit ball of $Q$ as $t \to 0$.
    However, it has remained unclear how to exploit the interplay between distinct rigid subalgebras not in specified relative position.

    We show that in fact, any diffuse subalgebra which is rigid with respect to a mixing s-malleable deformation is contained in a subalgebra which is uniquely maximal with respect to being rigid. In particular, the algebra generated by any family of rigid subalgebras that intersect diffusely must itself be rigid with respect to that deformation. The case where this family has two members was the motivation for this work, showing for example that if $G$ is a countable group with $\beta^{1}_{(2)}(G) > 0$, then $L(G)$ cannot be generated by two property $(T)$ subalgebras with diffuse intersection; however, the result is most striking when the family is infinite.
\end{abstract}

\section{Introduction}\label{S:Intro}

Sorin Popa's breakthrough deformation/rigidity theory, initiated in his seminal works \cite{Po01a, PopaL2Betti, PopaStrongRigidity, PopaStrongRigidtyII},
has established the paradigm for how rigidity properties of subalgebras can be exploited in the presence of deformability of an ambient von Neumann algebra, leading to a number of landmark results in directions that had before remained almost completely intractable. We refer the reader to \cite{PopaICM,Va06a,Va10a,Io12b,Io17c} for exposition of this remarkable progress, including the resolution of many long-standing problems.

Given a tracial von Neumann algebra $(M, \tau)$, let $\Aut(M)$ denote the group of trace preserving $*$-automorphisms of $M$, and as usual let $\|x\|_2 = \sqrt{\tau(x^*x})$ for $x \in M$.
We recall the foundational notions of malleable and s-malleable deformations, due to Popa and developed throughout his work in \cite{Po01a, PopaStrongRigidity, PopaStrongRigidtyII, PopaCohomologyOE, PopaCoc}.

\begin{defn}[Popa]%[Popa, cf.~{\cite[Definition 4.3]{PopaCoc}}]
Let $M \subseteq \widetilde M$ be tracial von Neumann algebras.
\begin{enumerate}
\item A {\it malleable deformation $\alpha$ of $M$ inside $\widetilde M$} is an action $\alpha: \mathbb{R} \to \Aut(\widetilde M)$ such that $\|\alpha_t(x) - x\|_2 \to 0$ as $t \to 0$ for each $x \in \widetilde M$.

\item An {\it s-malleable deformation $(\alpha, \beta)$ of $M$ inside $\widetilde M$} is $\alpha$ as above, together with $\beta \in \Aut(\widetilde M)$ satisfying $\beta|_{M} = \id,$ $\beta^{2}=\id,$ and $\beta \alpha_t = \alpha_{-t}\beta$  for all $t \in \mathbb{R}.$

\end{enumerate}
\end{defn}

The power of deformability lies in Popa's profound discovery that it allows one to ``locate" subalgebras which have properties forcing the deformation to converge to the identity uniformly over their unit balls as $t \to 0$.
For example, this applies to all subalgebras with Property (T),
a fact used to great effect from the
beginning of the subject \cite{Po01a, PopaL2Betti}. More generally, given a malleable deformation $\alpha_{t}\colon \widetilde{M}\to \widetilde{M}$ of a tracial von Neumann algebra $(M,\tau),$ a von Neumann subalgebra $Q \subseteq M$ will be said to be {\it rigid with respect to $\alpha$} (or {\it $\alpha$-rigid}) if the above convergence holds; that is, if
    \begin{align*}
        \epsilon_t(Q) \stackrel{{\rm def}}{=} \sup_{x \in (Q)_1}\|\alpha_t(x) - x\|_2 \quad\text{ has }\quad \epsilon_t(Q) \to 0 \quad\text{ as }\quad t \to 0.
        \end{align*}

We will write ${\rm Rig}(\alpha)$ for the collection of all $\alpha$-rigid von Neumann subalgebras of $M$, and ${\rm MaxRig}(\alpha)$ for the collection of $P \in {\rm Rig}(\alpha)$ which are maximal in ${\rm Rig}(\alpha)$ with respect to inclusion.
It is perhaps surprising that this last definition is worth making, as it is not clear that members of ${\rm MaxRig}(\alpha)$ arise in enough situations to be of interest.
%Generally, uniform convergence is not a property that one expects to pass from an increasing chain of sets to its union.
The most striking result of this paper, however, is that in the setting of s-malleable deformations this behavior is in fact typical:

\begin{thm}\label{T:main theorem intro}
Let $(\alpha, \beta)$ be an  s-malleable deformation of tracial von Neumann algebras $M \subseteq \widetilde M$.
Then any $Q \in {\rm Rig}(\alpha)$ with $Q' \cap \widetilde M \subseteq M$ is contained in a unique $P \in {\rm MaxRig}(\alpha)$.

\end{thm}
\noindent Note that by a fundamental insight of Popa \cite[Section 3]{PopaStrongRigidity}, when $L^2(\widetilde M) \ominus L^2(M)$ is a mixing $M$-$M$ bimodule, the condition $Q' \cap \widetilde M \subseteq M$ is automatically satisfied whenever $Q$ is diffuse.

There are two main implications from Theorem \ref{T:main theorem intro}. First, the existence of such $P\in \MaxR(\alpha)$ establishes that the strong operator topology closure of an increasing chain of $\alpha$-rigid von Neumann subalgebras containing $Q$ as above is again $\alpha$-rigid. Second, from the uniqueness of $P$, it follows that if $Q_1, Q_2 \in {\rm Rig}(\alpha)$ contain such a $Q$, then the subalgebra $Q_1 \vee Q_2$ they generate is itself $\alpha$-rigid.
It was in fact this second implication that was the motivation for this work, but the first is perhaps the more surprising of the two. Both are evident from the next result, from which Theorem \ref{T:main theorem intro} follows. We thank Stefaan Vaes for allowing us to use a proof of his which greatly simplified the one given in an earlier draft of this work and improved the estimate below.

\begin{thm}\label{T:concrete estimate intro}
Let $(\alpha, \beta)$ be an s-malleable deformation of tracial von Neumann algebras $M \subseteq \widetilde{M}$.

Then for any $Q_1, Q_2 \in {\rm Rig}(\alpha)$ with $(Q_1 \cap Q_2)' \cap \widetilde M \subseteq M$, and $t \in \mathbb{R}$, we have
\begin{align*}
    \epsilon_{2t}(Q_1 \vee Q_2) \le 24\epsilon_t(Q_1).
\end{align*}
\end{thm}

The proof of Theorem \ref{T:concrete estimate intro} relies on several pivotal discoveries due to Popa \cite{PopaStrongRigidity, PopaCoc, PopaSG}. In order to explain, we briefly sketch the proof.

Set $Q=Q_{1}\vee Q_{2}.$ First, we rely on Popa's key insight that the involution $\beta$ allows one, by way of his transversality inequality (\cite[Lemma 2.1]{PopaSG}, see \Cref{thm:PopasTranservality} below), to prove that $Q$ is rigid by proving that $\alpha_{t}(Q)$ can be conjugated into $M$ by a unitary in $\widetilde{M}$ which is close to $1$ in $\|\cdot\|_{2}.$ In fact, the crucial discoveries of Popa in \cite[Theorem 4.4(ii)]{PopaStrongRigidity}, \cite[Lemma 4.6]{PopaCoc}, \cite[Theorem 4.1]{PopaSG} imply that the converse is true: if $B\leq M$ is rigid, then for small $t$ one can in fact find a unitary close to $1$ which implements $\alpha_{t}$ on $B$. So we know that $\alpha_{t}$ can be unitarily implemented on $Q_j$ for $j=1,2$ by unitaries in $\widetilde{M}$ which are close to $1.$ Since $(Q_{1}\cap Q_{2})'\cap \widetilde{M}\subseteq M,$ these unitaries differ by a unitary which is in $M.$ So we can unitarily conjugate $\alpha_{t}(Q_{1}\vee Q_{2})$ into $M$ by a unitary in $\widetilde{M}$ which is close to $1,$ and hence $Q$ is rigid.

A notable consequence of Theorem \ref{T:concrete estimate intro} is that any $Q \in {\rm Rig}(\alpha)$ with $Q' \cap \widetilde{M} \subseteq M$ has its rigidity controlled by {\it any} of its subalgebras $N$ satisfying $N' \cap \widetilde M \subseteq M$. Naturally, $Q$ itself then controls the rigidity of any rigid subalgebra containing it. This is worth stating precisely:
\begin{cor}\label{cor:defcontrolsubalgebra}
Let $(\alpha, \beta)$ be an s-malleable deformation of tracial von Neumann algebras $M \subseteq \widetilde M$.\nopagebreak

Then for any $Q \in {\rm Rig}(\alpha)$ with $Q' \cap \widetilde M \subseteq M$, and any $t \in \mathbb{R}$, we have
\begin{align*}
    \frac{1}{24}\hspace{-2.2ex}\sup_{\stackrel{N \geq Q}{N \in {\rm Rig}(\alpha)}} \hspace{-2.2ex}\epsilon_{4t}(N) \; \le \; \epsilon_{2t}(Q) \;\le\; 24\hspace{-2.7ex} \inf_{\stackrel{N \leq Q}{N' \cap \widetilde M \subseteq M}} \hspace{-2.7ex}\epsilon_t(N).
\end{align*}
\end{cor}

One of the major applications of Popa's deformation/rigidity theory is to the class of \emph{group von Neumann algebras}. In this context, cohomological and representation theoretic data play a significant role. Theorem \ref{T:main theorem intro} gives a new way to apply these tools  to the study of group von Neumann algebras. Suppose that $G$ is a countable, discrete group and $\pi\colon G\to \mathcal{O}(\mathcal{H}_{\R})$ is an orthogonal representation on a real Hilbert space $\mathcal{H}_{\R}.$ Given a cocycle $c\colon G\to \mathcal{H}_{\R}$ for $\pi,$ there is a natural way to ``exponentiate" $c$ to construct an  s-malleable deformation $\alpha_{c, t}\colon \widetilde{M} \to \widetilde{M}$ of the group von Neumann algebra $M=L(G)$ (see Section \ref{S:1-cohomology} for the precise details).
It can be directly established that $L^{2}(\widetilde{M})\ominus L^{2}(L(G))$ is a mixing $L(G)$-$L(G)$ bimodule if and only if $\pi$ is mixing, and that $L(G)$ is $\alpha_{c}$-rigid if and only if $c$ is bounded, i.e., inner.  In this context, Theorem \ref{T:main theorem intro} is analogous to the known fact that if a cocycle associated to a mixing representation is inner on two subgroups with infinite intersection, then it is inner on the group they generate. Indeed, much of the work in this paper is motivated by previous works in cohomology (e.g., \cite{PopaCoc, PopaCohomologyOE, PetersonThom, PetersonL2}).
Our methods  produce the following concrete corollary.

\begin{thm}\label{T:group case intro}
Let $G$ be a countable, discrete group, and let $\pi\colon G\to \mathcal{O}(\mathcal{H}_{\R})$ be an orthogonal representation where $\mathcal{H}_{\R}$ is a real Hilbert space. Suppose that $\pi$ is mixing and that $H^{1}(G,\mathcal{H}_{\R})\ne \{0\}.$ Then, for any pair $Q_{j},j=1,2$ of property (T) subalgebras of $L(G)$ such that $Q_{1}\cap Q_{2}$ is diffuse, we have that $Q_{1}\vee Q_{2}$ has infinite Jones index in $L(G).$

More generally, let $\mathcal{C}$ be the smallest class of von Neumann subalgebras of $L(G)$ that contains all property (T) subalgebras and is closed under the following operations:
\begin{itemize}
    \item if $Q\in \mathcal{C},$ then $W^{*}(\mathcal{N}_{M}^{wq}(Q)),W^{*}(q^{1}\mathcal{N}_{M}(Q)),W^{*}(wI_{M}(Q,Q))\in \mathcal{C},$ (see the discussion preceding Theorem \ref{T:wq-normalizer intro})
    \item if $Q_{1},Q_{2}\in \mathcal{C},$ and $Q_{1}\cap Q_{2}$ is diffuse, then $Q_{1}\vee Q_{2}\in \mathcal{C},$
    \item if $(Q_{j})_{j}$ is an increasing chain in $\mathcal{C},$ then ${\bigvee_{j}Q_{j}}\in \mathcal{C},$
    \item if $Q\in \mathcal{C}$ and $N$ is an intermediate von Neumann subalgebra between $Q$ and $M$ such that $Q\leq N$ has finite Jones index, then $N\in \mathcal{C}.$
\end{itemize}
Then $L(G)\notin \mathcal{C}.$

\end{thm}

If $\pi$ is the left regular representation and $G$ is nonamenable, then the assumption $H^{1}(G,\mathcal{H}_{\R})\ne \{0\}$ is equivalent to saying that the first $\ell^{2}$-Betti number of $G,$ denoted $\beta_{(2)}^{1}(G),$ is positive \cite{PetersonThom}.
Viewed from the lens of $\ell^{2}$-Betti numbers, Theorem \ref{T:group case intro} is analogous to results of Peterson and Thom (see \cite[Theorem 5.6, Corollary 5.8]{PetersonThom}). They show that if $\beta^{1}_{(2)}(G)>0$, then $G$ cannot be generated by two property (T) subgroups which have infinite intersection.
As another example, under the hypotheses of Theorem \ref{T:group case intro}, suppose that $Q_{1},\cdots,Q_{k}\leq L(G)$ are diffuse, have property (T), and that $[Q_{j+1},Q_{j}]=\{0\}$ for $j=1,\cdots,k-1.$ Then $Q_{1}\vee Q_{2}\vee \cdots \vee Q_{k}$ has infinite Jones index in $L(G).$ This is analogous to a result of Gaboriau in the group case, in the context of cost (see \cite[Crit\`eres VI.24.]{Gab1}). The case of group von Neumann algebras is also one situation where we can give concrete examples of ${\rm MaxRig}(\alpha).$ See Corollary \ref{C:max rigid subgroup} below.

Returning to the general context of malleable deformations, we remark that many of the properties we can establish for elements of ${\rm MaxRig}(\alpha)$ are analogous to properties of \emph{Pinsker algebras} in the context of entropy theory for measure-preserving systems.
To illustrate how this perspective interfaces with Popa's deformation/rigidity theory, we discuss the case of weak normalizers of rigid subalgebras here and refer the reader to Theorem \ref{T:main properties of max rig intro}, several of the parts of which have analogues in the entropy theory context, for other instances.

By work of Popa in \cite[Theorem 4.1, Theorem 4.4]{PopaStrongRigidity} (see also \cite[Section 6]{Va06a}) and Peterson \cite[Theorem 4.5]{PetersonL2}), it is well known that if $(\alpha,\beta)$ is an s-malleable deformation of an inclusion of tracial von Neumann algebras $M\subseteq \widetilde{M}$, then mixingness of $L^{2}(\widetilde{M})\ominus L^{2}(M)$ as an $M$-$M$ bimodule often allows one to upgrade certain properties of a diffuse von Neumann subalgebra $Q \subseteq M$ to various weak normalizers. For example, it is well known that $Q \in {\rm Rig}(\alpha)$ gives $W^{*}(\mathcal{N}_{M}(Q))\in {\rm Rig}(\alpha)$ as in \cite[Theorem 4.5]{PetersonL2}. In light of these results
and in light of the analogy between elements of ${\rm MaxRig}(\alpha)$ and Pinsker algebras in the context of $1$-bounded entropy/ergodic theory, one might suspect to be able to pass rigidity not just to normalizers,  but more generally von Neumann algebraic weak normalizers as in \cite{Hayes8}. Fortunately, this is  indeed the case.
Recall (cf. \cite{PopaStrongRigidity, PopaCohomologyOE, IPP, GalatanPopa}) that if $M$ is a von Neumann algebra and $N\leq M,$ then the wq-normalizer of $Q$ inside of $M,$ denoted $\mathcal{N}_{M}^{wq}(N)$ is defined by
\[\mathcal{N}_{M}^{wq}(N)=\{u \in \mathcal{U}(M):uNu^{*}\cap N\mbox{ is diffuse}\}.\]
This is analogous to the wq-normalizer for an inclusion $H\leq G$ of groups, defined by Popa \cite[Definition 2.3]{PopaCohomologyOE}:
\[\mathcal{N}_{G}^{wq}(H)=\{g\in G:gHg^{-1}\cap H\mbox{ is infinite}\}.\]
The wq-normalizer of groups has already seen many applications in context where cohomology of groups with values in a unitary representation are relevant (e.g., \cite{PopaCohomologyOE},\cite{PetersonThom}). The  wq-normalizer of von Neumann algebras also appears naturally in situations where cohomology of von Neumann algebras is concerned (see, for example, \cite{GalatanPopa}).
The \emph{one-sided quasi-normalizer} is the set of all $x\in M$ so that there are $x_{1},\cdots,x_{k}\in M$ with
\[Nx\subseteq \sum_{j=1}^{k}x_{j}N,\]
and is denoted by $q^{1}\mathcal{N}_{M}(N).$
We say that $N$ is \emph{one-sided quasiregular} in $M$ if $W^{*}(q^{1}\mathcal{N}_{M}(N))=M.$
The one-sided quasi-normalizer first appears in \cite{IzumiLongoPopa}, where they use the term ``discrete" to refer to a quasi-regular inclusion (see also \cite[Proposition 1.3.1]{PPEntropy}).
% It was further studied in \cite{FGS10}. 
The inspiration for the one-sided quasi-normalizer can be traced back to the work of Popa in \cite{PopaOrthoPairs}.
We can prove that, under the assumption that $L^{2}(\widetilde{M})\ominus L^{2}(M)$ is a mixing $M$-$M$ bimodule,  rigidity passes from a diffuse algebra to the von Neumann algebra generated by its wq-normalizer, as well as its one-sided quasi-normalizer. We can also upgrade rigidity to the weak intertwining space, $W^{*}(wI_{M}(N,N))$, defined by Popa in \cite{PopaMC, PopaWeakInter} (see also \cite{PopaCohomologyOE, IPP, GalatanPopa} for related notions). We recall the definition of the weak intertwining space in Section \ref{S:mixingness}.  

\begin{thm}\label{T:wq-normalizer intro}
Let $(\alpha,\beta)$ be an s-malleable deformation of tracial von Neumann algebras $M\subseteq \widetilde{M}.$ Suppose that $L^{2}(\widetilde{M})\ominus L^{2}(M)$ is a mixing $M$-$M$ bimodule. Then for any diffuse $Q\in {\rm Rig}(\alpha),$ we have that $W^{*}(\mathcal{N}_{M}^{wq}(Q)), W^{*}(q^{1}\mathcal{N}_{M}(Q)),W^{*}(wI_{M}(Q,Q)) \in {\rm Rig}(\alpha).$
\end{thm}

Theorem \ref{T:wq-normalizer intro} is motivated by striking results of Popa (see \cite[Corollary 0.2]{PopaCoc} and \cite[Lemma 2.4]{PopaCohomologyOE}) that upgrade cocycle untwisting on a subgroup to untwisting on its wq-normalizer and unitary intertwining of a subalgebra of $\widetilde{M}$ back into $M$ to weak versions of normalizers of that subalgebra, in the context of mixing group actions (see \cite[Theorem 4.1, Theorem 4.4]{PopaStrongRigidity},  \cite[Section 6]{Va06a}, \cite[Theorem 6.5]{SorinStefaanGeneralizedBern}).
Despite the elegance and generality of these results, the results of this paper constitute the first instance where one can upgrade rigidity, and not just intertwining, to the von Neumann algebra generated by  ``weak versions" of the normalizer (specifically the wq-normalizer, the one-sided quasi-normalizer, and the weak intertwining space),
in the general abstract setting of a mixing s-malleable deformation of a tracial von Neumann algebra. 
% In fact, we do not know of a proof of Theorem \ref{T:wq-normalizer intro} which avoids the use of maximal rigid subalgebras. 
Part of the difficulty is that none of these ``weak normalizers" are groups.
We refer the reader to the discussion following Corollary \ref{C:upgrading rigidity to all the normalizers}. We remark that if $M\leq \widetilde{M}$ is \emph{coarse} in the sense of \cite{PopaWeakInter}, and if $Q\leq M$ is $\alpha$-rigid, then so is the von Neumann algebra generated by the \emph{singular subspace of $Q\leq M$} as defined in \cite{Hayes8} (see Section \ref{S:mixingness} for the relevant definitions). This is a direct analogue of \cite[Theorem 1.5]{Hayes8}.

Lastly, we present applications to $L^{2}$-rigidity of II$_1$ factors initiated by Peterson \cite{PetersonL2}.
The notion of $L^{2}$-rigidity is significant because it is one of a very few cohomological properties which is well-defined for all II$_1$ factors and which implies (and is implied by) interesting structural properties. For example, nonamenable II$_1$ factors with  property Gamma or property (T) are $L^2$-rigid \cite{PetersonL2}.

Using the dilation theorem of Dabrowski (see \cite[Theorem 20, Proposition 31]{Dabrowski} as well as \cite[Theorem 3.1]{DabrowskiIoana}), we are able to give an analogue of our main results in the context of $L^{2}$-rigidity.

\begin{thm} Let $M$ be a II$_1$ factor, and let $\LRig(M)$ denote the collection of all diffuse von Neumann subalgebras of $M$ which are $L^2$-rigid in the sense of Definition \ref{L2rigid}. If $Q\in \LRig(M)$, then $Q$ is contained in a unique maximal $P\in \LRig(M)$.

Moreover, $\LRig(M)$ is closed under the following operations:
\begin{itemize}
    \item if $Q\in \LRig(M),$ then $W^{*}(\mathcal{N}_{M}^{wq}(Q))\in \LRig(M),$
    \item if $Q_{1},Q_{2}\in \LRig(M),$ and $Q_{1}\cap Q_{2}$ is diffuse, then $Q_{1}\vee Q_{2}\in \LRig(M),$
    \item if $(Q_{\alpha})_\alpha$ is a chain in $\LRig(M),$ then $\bigvee_{\alpha}Q_{\alpha}\in \LRig(M),$
    \item if $Q\in \LRig(M)$ and $N$ is an intermediate von Neumann subalgebra between $Q$ and $M$ such that $Q\leq N$ has finite Jones index, then $N\in \LRig(M).$
\end{itemize}

\end{thm}

\noindent In particular, this establishes one additional permanence property of $L^{2}$-rigidity which is analogous to a permanence property we have for groups with vanishing first $\ell^{2}$-Betti number. For example, it is known that if two subgroups of a fixed group have vanishing first $\ell^{2}$-Betti number and infinite intersection, then the group they generate has vanishing first $\ell^{2}$-Betti number.

\subsection*{Structure of the paper} Aside from the introduction, this paper has six sections.  In Section \ref{S:Prelim}, we remind the reader of important notions in von Neumann algebras, in particular, Popa's notion of an s-malleable deformation of a semifinite von Neumann algebra. In Section \ref{S:mainproof}, we establish the crucial estimate  Theorem \ref{T:concrete estimate intro} from which we derive Theorem \ref{T:main theorem intro}. The discussion found in Section \ref{S:1-cohomology} is devoted to the case where the von Neumann algebra and the deformation arise from group-theoretic considerations.
Section \ref{S:permanence} is devoted to investigating permanence properties of maximal rigid subalgebras, particularly their behavior under tensor products and amplifications. The consequences of our results in the presence of mixingness are detailed in Section \ref{S:mixingness}.% The
We conclude this paper by exploring connections of our results to $L^2$-rigidity. In particular, we use our techniques in combination with Dabrowski's dilation theorem to give a new approach to a conjecture of Peterson and Thom \cite{PetersonThom}.
\subsection*{Acknowledgements}

The authors thank the American Institute of Mathematics for hosting the workshop ``Classification of Group von Neumann Algebras'' during which the initial results that led to this project were completed. The authors are indebted to Jesse Peterson for posing a problem at the workshop which initiated the work here and whose solution is contained in Theorem \ref{T:group case intro}. Part of this work was done during the
program ``Quantitative Linear Algebra'' at the Institute for Pure and Applied Mathematics. The authors thank IPAM for its hospitality. The authors kindly thank Sorin Popa and Stefaan Vaes for suggesting many improvements on an earlier version of this paper.

\section{Preliminaries}\label{S:Prelim}
We begin by recalling several notions involving semifinite von Neumann algebras.
Let $M\subseteq B(\mathcal{H})$ be a von Neumann algebra with unit $1=\operatorname{id}_{\mathcal{H}} $.    The set $$ M':=\{y\in B(\mathcal{H}) : my=ym \, \text{for all }m\in M\} $$ will denote the \emph{commutant} of $M$ inside of $B (\mathcal{H})$. $N\leq M $ will be used to denote that $N\subseteq M $ is a unital inclusion of  von Neumann algebras, and the set $N'\cap M $ will be called the \emph{relative commutant of $N$ inside $M$}. Whenever $S\subseteq M $ (or $\{P_j\}_{j\in \mathcal{J}}  $ is a collection of subalgebras of $M$), $W^*(S) $ is the smallest unital von Neumann subalgebra of $M$ containing $S$ ($\bigvee_{j\in \mathcal{J}} P_j$ containing $ \bigcup_{j\in\mathcal{J}} P_j$). Letting $S$ be a self-adjoint subset of $M$ containing $1$, von Neumann's Bicommutant Theorem gives $W^*(S)=M $ if and only if $S'':=(S')'=M $. $\mathcal{U}(M) $ is the group of unitaries of $M$,  $\mathcal{P}(M)$ are the projections of $M$, $Z(M)=M\cap M' $ is the center of $M$ and $(M)_1 $ is the unit ball of $M$.  When $(M,\tau)$ is a II$_1$ factor, $s>0 $ is a scalar, $M^s:=pM\Bar{\otimes } B(\ell^2(\N))p   $ is a corner of  $M\Bar{\otimes } B(\ell^2(\N)) $ such that $p\in \mathcal{P}(M\Bar{\otimes } B(\ell^2(\N))) $ is a projection with $\tau\bar\otimes \operatorname{Tr}(p)=s $  and $\operatorname{Tr} $ is the usual semifinite trace on $B(\ell^2(\N))$.

A \emph{tracial weight} on a von Neumann algebra $M$ is a map $w:M_+\to [0,\infty]$ which is additive on $M_+$, homogeneous over $ \R_{+}$ (with the convention that $0\cdot \infty=0$), normal, and satisfies $ w(x^{*}x)=w(xx^{*})$ for all $x\in M_+ $. $w$ is \emph{semifinite}  if for every $x\in M_+ $, there exists a non-zero $0< y\leq x $ so that $w(y)<\infty $.
In the  general situation where $M$ is a semifinite von Neumann algebra with faithful tracial weight $ w$, we may linearly extend $w$ to a function $\tau $ on the $\C $ linear span of $\{x\in M_{+}: w(x)<\infty\} $.  $\tau $ is then a semifinite trace on $M$ and we denote this pairing by $(M,\tau) $, and call $(M,\tau)$ a \emph{semifinite tracial von Neumann algebra}.
Should
 $ w(1_M)<\infty$, we  assume $w(1_M)=1$ and in this case $\tau$  is defined on all of $M$ and $(M,\tau)$ is called a \emph{tracial von Neumann algebra}.
 In the case where $M$ is a \emph{factor}, i.e.~$M\cap M'\cong \C $, any semifinite trace $\tau $ on $M$ is unique up to multiplication by a positive scalar.

 Fixing a semifinite tracial von Neumann algebra $(M,\tau) $, $L^2(M,\tau) $ is the Hilbert space completion of $M_{0} :=\{x\in M : \tau(x^*x)<\infty\}$ with respect to the sesquilinear form defined by $\langle x,y\rangle :=\tau(y^*x) $ for all $x,y\in M_0$. For all $x\in M$, let $\|x\|_2 = \tau(x^*x)^{1/2} \in [0, \infty]$.
Choosing a von Neumann subalgebra $N\subseteq M$, so that $\tau\big|_{N}$ is still semifinite, $\E_N:M\to N $ will denote the unique $ \tau$-preserving conditional expectation.

For our purposes, it will be important to recall Popa's notion of an s-malleable deformation in the semifinite setting.

%, cf.~{\cite[Definition 4.3]{PopaCoc}}
\begin{defn}[Popa, cf. \cite{Po01a, PopaStrongRigidity, PopaStrongRigidtyII, PopaCohomologyOE, PopaCoc}]\label{D:smalleable semifinite}
Let $(M,\tau)$ be  a semifinite tracial von Neumann algebra. An s-malleable deformation of $M$ is a tuple $((\alpha_{t})_{t\in \R},\beta,\widetilde{M},\widetilde{\tau})$ where
\begin{itemize}
    \item $(\widetilde{M},\widetilde{\tau})$ is a semifinite tracial von Neumann algebra and $M\leq\widetilde{M}$ in a trace-preserving way,
    \item $t\mapsto \alpha_{t}$ is a homomorphism,
     \item $\alpha_{t}$ is a trace-preserving automorphism of $\widetilde{M}$ for each $t\in \R,$ and for every $x\in \widetilde{M}$ the map $t\mapsto \alpha_{t}(x)$ is continuous if we give $\widetilde{M}$ the strong topology,
     \item $\beta$ is a trace-preserving automorphism of $\widetilde{M}$ with $\beta^{2}=\id,$ $\beta\big|_{M}=\id$ and $\beta \alpha_{t}=\alpha_{-t}\beta$ for all $t \in \mathbb{R}$.
\end{itemize}

\end{defn}
Though we use a different notation $\widetilde{\tau}$ for the trace on $\widetilde{M},$ we will often in proofs just write the trace on $\widetilde{M}$ as $\tau$ for ease of notation. As the embedding $M\leq \widetilde{M}$ is trace-preserving, this will cause no confusion.

We will not actually need to use that $\beta^{2}=\id,$ but this is a standard assumption and we may typically assume it without much loss of generality. In most arguments, one does not usually need to work with all of $\widetilde{M},$ and it is sufficient to understand how the deformation behaves on the subalgebra $\bigvee_{t\in \R}\alpha_{t}(M)$. It is easy to see that if we only assume the first two items of Definition \ref{D:smalleable semifinite}, then $\beta^{2}\big|_{\bigvee_{t\in \R}\alpha_{t}(M)}=\id.$ So removing the assumption $\beta^{2}=\id$ does not typically lead to different results.

We also remark that we do not really need to assume that $\beta$ is defined on all of $\widetilde{M}$, and instead it is sufficient to assume that $\beta$ is defined on $\bigvee_{t\in \R}\alpha_{t}(M)$ and satisfies $\beta\big|_{M}=\id,\beta\alpha_{t}=\alpha_{-t}\beta$ there. An exercise in using the GNS construction shows that the existence of a $\beta$ on $\bigvee_{t\in \R}\alpha_{t}(M)$ which satisfies the first two assumptions of Definition \ref{D:smalleable semifinite} is equivalent to the following moment condition: for all $k\in \N,$ all $x_{1},\cdots,x_{k}\in M,$ and all $t_{1},\cdots,t_{k}\in \R$ we have:
\[\tau(\alpha_{t_{1}}(x_{1})\alpha_{t_{2}}(x_{2})\cdots \alpha_{t_{k}}(x_{k}))=\tau(\alpha_{-t_{1}}(x_{1})\alpha_{-t_{2}}(x_{2})\cdots \alpha_{-t_{k}}(x_{k})).\]
In fact, applying $\alpha_{s}$ reduces this moment condition to one only involving positive times: for all $k\in \N,$ all $x_{1},\cdots,x_{k}\in M,$ all $s\in (0,\infty)$ and all $t_{1},\cdots,t_{k}\in (0,s):$
\[\tau(\alpha_{t_{1}}(x_{1})\alpha_{t_{2}}(x_{2})\cdots \alpha_{t_{k}}(x_{k}))=\tau(\alpha_{s-t_{1}}(x_{1})\alpha_{s-t_{2}}(x_{2})\cdots \alpha_{s-t_{k}}(x_{k})).\]
In the classical setting (i.e. the case that $M,\bigvee_{t\in \R}\alpha_{t}(M)$ are abelian), this is known as saying that the flow $\alpha_{t}$ is \emph{time-reversible}.

The following is well known, but we include it to ease any concerns the reader may have as to what it means for a one-parameter group of trace preserving automorphisms of a semifinite tracial von Neumann algebra to be continuous.

\begin{prop}
Let $(M,\tau)$ be a semifinite tracial von Neumann algebra and $\alpha_{t}\colon M\to M$ for $t\in \R$ a one-parameter group of trace-preserving automorphisms. The following are equivalent:
\begin{enumerate}[(a)]
\item $t\mapsto \alpha_{t}$ is pointwise strong operator topology continuous; \label{I:1pSOT}
\item $t\mapsto \alpha_{t}$ is pointwise continuous in the strong$^{*}$ topology; \label{I:1pSOT*}
\item  $t\mapsto \alpha_{t}\big|_{M\cap L^{2}(M)}$ is pointwise $\|\cdot\|_{2}$-continuous. \label{I:ip2}
\end{enumerate}
\end{prop}

\begin{proof}
(\ref{I:1pSOT}) implies (\ref{I:1pSOT*}): Obvious from the fact that $\alpha_{t}(x^{*})=\alpha_{t}(x)^{*}.$

(\ref{I:1pSOT*}) implies (\ref{I:ip2}):
Since $t\mapsto \alpha_{t}$ is a homomorphism, it suffices to show pointwise $\|\cdot\|_{2}$-continuity at $t=0.$
Since $\|\alpha_{t}(x)\|_{2}=\|x\|_{2}$ for all $x\in M\cap L^{2}(M),$ and $M\cap L^{1}(M)$ is $\|\cdot\|_{2}$-dense in $M\cap L^{2}(M),$ it suffices to show that for every $x\in M\cap L^{1}(M)$ we have that $t\mapsto \alpha_{t}(x)$ is $\|\cdot\|_{2}$-continuous at $t=0$.

Since $x\in M\cap L^{1}(M),$ we may write $x=ab$ for $a,b\in M\cap L^{2}(M).$ Then:
\[\|\alpha_{t}(x)-x\|_{2}^{2}=2\|x\|_{2}^{2}-2\Re(\ip{\alpha_{t}(x),ab})=2\|x\|_{2}^{2}-2\Re(\ip{\alpha_{t}(x)b^{*},a}).\]
By SOT-continuity at $t=0,$ we have that
\[\lim_{t\to 0}\|\alpha_{t}(x)-x\|_{2}^{2}=2\|x\|_{2}^{2}-2\Re(\ip{xb^{*},a})=2(\|x\|_{2}^{2}-\Re(\ip{x,x}))=0.\]

(\ref{I:ip2}) implies (\ref{I:1pSOT}):
As in (\ref{I:1pSOT*}) implies (\ref{I:ip2}), it suffices to show pointwise SOT-continuity at $t=0.$ Fix an $x\in M.$ Since $\|\alpha_{t}(x)\|\leq \|x\|$ for all $t\in \R,$ it suffices to show that $\|\alpha_{t}(x)b-xb\|_{2}\to 0$ for every $b\in M\cap L^{2}(M).$

Fix a $b\in M\cap L^{2}(M)$ and let $\varepsilon>0.$ Apply Kaplansky's density theorem to find a $y\in M\cap L^{2}(M)$ so that $\|xb-yb\|_{2}<\varepsilon$ and $\|y\|\leq\|x\|.$ Then:
\begin{align*}
    \|\alpha_{t}(x)b-xb\|_{2}&\leq \varepsilon+\|\alpha_{t}(x-y)b\|_{2}+\|(\alpha_{t}(y)-y)b\|_{2}\\
    &\leq \varepsilon+\|(x-y)\alpha_{-t}(b)\|_{2}+\|b\|\|\alpha_{t}(y)-y\|_{2}\\
    &\leq 2\varepsilon+\|(x-y)(\alpha_{-t}(b)-b)\|_{2}+\|b\|\|\alpha_{t}(y)-y\|_{2}\\
    &\leq 2\varepsilon+2\|x\|\|\alpha_{-t}(b)-b\|_{2}+\|b\|\|\alpha_{t}(y)-y\|_{2}.
\end{align*}
Since $\alpha_{t}\big|_{M\cap L^{2}(M)}$ is pointwise $\|\cdot\|_{2}$-continuous we can let $t\to 0$ to see that
\[\limsup_{t\to 0}\|\alpha_{t}(x)b-xb\|_{2}\leq 2\varepsilon.\]
Since $\varepsilon$ is arbitrary, this completes the proof. \qedhere

\end{proof}

The notion of being rigid with respect to a deformation is crucial for our discussion. Given $((\alpha_{t})_{t\in \R},\widetilde{M})$ as in Definition \ref{D:smalleable semifinite}, and a projection $p\in M$ with $\tau(p)<\infty,$ we say that $Q\leq pMp$ is \emph{$\alpha$-rigid} if
\[\sup_{x\in Q:\|x\|\leq 1}\|\alpha_{t}(x)-x\|_{2}\to_{t\to 0}0.\]

One of Popa's key innovations, which will be crucial in our investigation, is the following property of s-malleable deformations, which shows that we may control how fast they converge to the identity in terms of how fast they move elements close to \emph{some} element of $M.$ The proof follows by the same argument as in \cite[Lemma 2.1]{PopaSG}.

\begin{thm}[Popa's Transversality Inequality]\label{thm:PopasTranservality}
Let $M\leq\widetilde{M}$ be a trace-preserving inclusion of semifinite von Neumann algebras. Let	$\alpha:\R\to \operatorname{Aut}(\widetilde{M}) $, $\beta\in \operatorname{Aut}(\widetilde{M}) $ be an s-malleable deformation of $M$ inside $\widetilde{M}$. Then for all $x\in M\cap L^{2}(M)$ and $t\in \R$,
\begin{align*}
\| \alpha_{2t}(x)-x \|_2\leq 2\|\alpha_t(x)- \mathbb{E}_M(\alpha_t(x))  \|_{2}.
\end{align*}
\end{thm}

\section{Proof of The Main Result}\label{S:mainproof}

The main goal of this paper is the study of von Neumann subalgebras which are rigid with respect to an s-malleable deformation and maximal with respect to inclusion among rigid subalgebras. For this reason, we make the following definition.
\begin{defn}
Let $(\widetilde{M},\alpha,\beta)$ be an s-malleable deformation of $M.$ Given $p\in \mathcal{P}(M)$ with $p\ne 0,$ we say that an $\alpha$-rigid $Q\leq pMp$ is \emph{maximal rigid with respect to $\alpha$} if whenever $P\leq pMp$ is $\alpha$-rigid and $P\supseteq Q,$ then $P=Q.$ If $\alpha$ is clear from the context, then we will often just say that $P$ is \emph{maximal rigid}.
\end{defn}

In this section we prove Theorem \ref{T:main theorem intro},
% which shows that rigid envelopes exist in many situations.
%Theorem \ref{T:main theorem intro}
which is the main result by which we obtain maximal rigid subalgebras. We actually handle the more general situation where the deformation in question is defined on a semifinite von Neumann algebra, and the subalgebra in question is a subalgebra of a \emph{finite trace} corner. Part of the motivation is this. Fix an s-malleable deformation $\alpha_{t,0}\colon \widetilde{M}_{0}\to\widetilde{M}_{0}$ of a tracial von Neumann algebra $(M_{0},\tau_{0}).$ Consider $\widetilde{M}=\widetilde{M}_{0}\overline{\otimes}B(\ell^{2}(\N)),$ $M=M_{0}\overline{\otimes}B(\ell^{2}(\N)),$ and $\alpha_{t}=\alpha_{t,0}\otimes \id_{B(\ell^{2}(\N))}.$ If $M_{0}$ is a factor, then a finite trace corner of $M$ is simply an amplification of $M_0.$ So working in the semifinite context will make it much clearer how maximal rigid subalgebras behave under amplification. Our work  applies to the situation where $M_{0}$ is not a factor, and so we  have a clean picture for ``generalized amplifications."

For the sake of exposition, it will be helpful to explicitly give a name to a maximal rigid subalgebra containing a given subalgebra.
\begin{defn}\label{D:rigid envelope}
Let $(M,\tau)$ be a semifinite tracial von Neumann algebra, and $(\widetilde{M},\widetilde{\tau})$ a semifinite tracial von Neumann algebra so that $M\leq \widetilde{M}$ in a trace-preserving way. Let $\alpha_{t}\colon \widetilde{M}\to \widetilde{M}$ be strongly continuous one-parameter group in $\Aut(\widetilde{M},\widetilde{\tau}).$ Suppose that $p\in M$ with $\tau(p)<\infty,$ and that $Q\leq pMp$ is $\alpha$-rigid. We say  that $P\leq pMp$ is a \emph{rigid envelope of $Q$ with respect to $\alpha$} if
\begin{itemize}
    \item $P$ is $\alpha$-rigid
    \item $P\supseteq Q,$
    \item if $N\leq pMp$ is $\alpha$-rigid and $N\supseteq Q,$ then $N\subseteq P.$
\end{itemize}
\end{defn}

A few comments are in order about the definition. The first is that it may not be the case that rigid envelopes exist, though the main results of this section show that in many natural cases they do. Second, it should be emphasized that the inclusion $Q\subseteq P$ is \emph{unital} by assumption. This is primarily because we want rigid envelopes to behave properly under compressions and amplifications. Moreover, if we were to allow  the inclusion $Q\subseteq P$ to be nonunital, then we could easily have situations where $Q$ has a rigid envelope in our sense but not under nonunital inclusions. For example, if $\widetilde{\tau}$ is finite, and $P$ is maximal rigid subalgebra of $M,$ and if $1\in P,$ and $P\supseteq Q,$ then $uPu^{*}$ is also maximal rigid and contains $Q$ for any $u\in \mathcal{U}(M)$ with $uQu^{*}=Q.$ For example, we could take $u=u_{1}+u_{2}$ where $u_{1}\in \mathcal{U}(Q),$ and $u_{2}\in \mathcal{U}((1-p)M(1-p)).$ In particular, if $p\ne 0,1$ there are \emph{many} maximal rigid subalgebras which nonunitally contain $Q.$ See Section \ref{S:compressions/amplifications}  for situations in which we can describe all  maximal rigid subalgebras $P$ which contain a given rigid $Q,$ and so that the inclusion of $Q$ into $P$ is not unital.

We use the following notation.
Let $(N,\tau_{N}),(M,\tau_{M})$ be semifinite tracial von Neumann algebras. Given a normal linear map $\phi\colon N\to M,$ we let $\|\phi\|_{\infty,2}=\sup_{x\in N:\|x\|\leq 1}\|\phi(x)\|_{2},$ which we allow to be $+\infty$ if the supremum is not finite. Note that when $\|\cdot\|_{\infty,2}$ is bounded, it can be viewed as the operator norm of $\phi$ regarded as a map $N\to M\cap L^{2}(M),$ so it can be thought of as an $L^{\infty}$--$L^{2}$ norm. Observe that if $\alpha_{t}\colon \widetilde{M}\to \widetilde{M}$ is an s-malleable deformation of $M,$ if $Q\leq pMp$ with $p\in \mathcal{P}(M),$ and $\tau(p)<\infty,$ then $Q$ is $\alpha$-rigid if and only if
\[\lim_{t\to 0}\|(\alpha_{t}-\id)\big|_{Q}\|_{\infty,2}=0.\]

The proof of the main theorem will make use of the following fundamental result due to Popa (Cf. \cite[Proof of Theorem 4.4(ii)]{PopaStrongRigidity}), one of his major discoveries on the power of s-malleability. For the reader's convenience, we reproduce
his well known proof.

\begin{prop}[Popa]\label{L:intertwining on big sets}
Let $(\widetilde{M},\tau)$ be a semifinite tracial von Neumann algebra and $M\subseteq \widetilde{M}$ a von Neumann subalgebra so that $\tau\big|_{M}$ is semifinite.  Let $p$ be a  projection in $M$ with $\tau(p)<\infty.$ Suppose that $\alpha_{t}\colon \widetilde{M}\to \widetilde{M}$ is an s-malleable deformation,  that $ Q\leq pMp$  is $\alpha$-rigid, and that $Q'\cap p\widetilde{M}p\subseteq pMp.$ Let $q\in \mathcal{P}(Q'\cap pMp)$ and $t\in \R.$ Then there exists a partial isometry $v\in\widetilde{M}$ so that $v^{*}v=q,$ and $vv^{*}=\alpha_{t}(q),$ and $\alpha_{t}(x)v=vx$ for all $x\in Q.$

\end{prop}

\begin{proof}

It is enough to show that whenever $q$ as in the statement of the proposition is nonzero, there exists a partial isometry $v_0\in \widetilde{M}$ and a nonzero projection $q_{0}\in Q'\cap pMp$ so that
\[q_{0}\leq q,\quad v_0^{*}v_0=q_{0},\quad \text{and} \quad  v_0v_0^{*}=\alpha_{t}(q_{0}).\]
The required $v$ is then obtained either as $v = 0$ in the case where $q = 0$, or as the sum of a maximal orthogonal family of partial isometries obtained as above when $q \ne 0$.

Since $Q$ is rigid, we may fix  $s$ of the form $t/n$ for some integer $n\geq 1,$ and so that the unique element $\xi$ of minimal $\|\cdot\|_{2}$-norm in the $\|\cdot\|_{2}$-closed convex hull of $\{\alpha_{s}(u)u^{*}:u\in \mathcal{U}(Qq)\}$ is not zero. By the uniqueness of $\xi$ and the fact that $q \in Q'$ we then have that
\[\alpha_{s}(x)\xi=\xi x\mbox{ for all $x\in Q$.}\]
Let $\xi=w|\xi|$ be the polar decomposition of $\xi,$ and let $q_{0}=w^{*}w.$ Since $Q'\cap p\widetilde{M}p\subseteq M,$ we obtain that $|\xi|,q_{0}\in Q'\cap pMp$. For all $u \in \mathcal{U}(Qq)$ we have
\[\xi = \alpha_s(u)\xi u^* = \alpha_s(u)w|\xi|u^* =  \alpha_s(u)wu^*|\xi|
\]
and hence $\alpha_s(u)wu^* = w$ by the uniqueness of the polar decomposition, so that $\alpha_{s}(x)w=wx$ for all $x\in Q$ (again using that $q \in Q'$).

Since $\beta\big|_{M}=\id,$ we have for every $u\in \mathcal{U}(Qq)$
\[\alpha_{s}(\beta(\alpha_{s}(u)u^{*}))=\alpha_{s}(\alpha_{-s}(u)u^{*}))=u\alpha_{s}(u)^{*},\]
so $\alpha_{s}(\beta(\xi))=\xi^{*}.$ So $\alpha_{s}(q_0)=\alpha_{s}(\beta(q_0))$ is the right support of $\xi^{*},$ and hence $ww^{*}=\alpha_{s}(q_0).$ Now setting
\[v_0=\alpha_{(n-1)s}(w)\alpha_{(n-2)s}(w)\cdots \alpha_{3s}(w)\alpha_{2s}(w)\alpha_{s}(w)w,\]
it is then easy to verify that $v_0$ has the desired properties. \qedhere
\end{proof}

For the sake of clarity, we isolate below a consequence of combining two foundational results of Popa, namely the preceding proposition and his transversality inequality. The main theorem will follow quickly from this. 

\begin{prop}\label{P:two sided estimate intertwiner}
Let $(M,\tau)$ be a semifinite tracial von Neumann algebra and $\alpha_{t}\colon\widetilde{M}\to\widetilde{M}$ an s-malleable deformation of $(M,\tau)$. Fix a projection $p\in M$ with $\tau(p) < \infty$, and for $t \in \R$ let $\mathcal{V}_{t}$ denote the set of partial isometries in $\widetilde{M}$ with $v^{*}v=p,vv^{*}=\alpha_{t}(p)$. 

Then for any $\alpha$-rigid $Q \le pMp$ with $Q'\cap p\widetilde{M}p\subseteq M$, the quantities
\begin{align*}
\varepsilon_{t}(Q) &= \|(\alpha_{t}-\id)\big|_{Q}\|_{\infty,2}, \\
\delta_t(Q) &= \inf\{\|v-p\|_{2}: v \in \mathcal{V}_{t},\, vx=\alpha_{t}(x)v \mbox{ for all $x\in Q$}\}, \quad\text{and} \\
\gamma_t(Q) &= \inf\{\|v-p\|_{2}: v \in \mathcal{V}_{t},\, v^*\alpha_t(Q)v \subseteq M\}
\end{align*}
satisfy
\begin{align*}
\frac{1}{4}\varepsilon_{2t}(Q) \le \gamma_t(Q) \le \delta_t(Q) \le 6\varepsilon_t(Q).
\end{align*}
\end{prop}

\begin{proof}
Throughout the proof, we use that as a consequence of traciality we have $\|v-p\|_{2}=\|v-\alpha_{t}(p)\|_{2}$ for every $v\in\mathcal{V}_{t},$ which is a computation that we leave as an exercise to the reader. %If $v\in \mathcal{V}_{t}$ and $vx=\alpha_{t}(x)v$ for all $x\in Q,$ then for every $x\in Q$:
%\[\|\alpha_{t}(x)-x\|_{2}\leq \|v-\alpha_{t}(p)\|+\|v+p\|_{2}+\|\alpha_{t}(x)v-vx\|_{2}=2\|v-p\|_{2}.\]
%So we have that $\varepsilon_{t}(Q)\leq 2\delta_{t}(Q).$

To establish the first inequality, take any $v \in \mathcal{V}_t$ such that $v^*\alpha_t(Q)v \subseteq M$. Then by Popa's transversality inequality,
for any $x \in Q$ with $\|x\|\le 1$ we have
\begin{align*}
    \|\alpha_{2t}(x) - x\|_2 \le  2\|\alpha_{t}(x)-\E_{M}(\alpha_{t}(x))\|_{2}
    \leq 2\|\alpha_{t}(x)-v^{*}\alpha_{t}(x)v\|_{2}
    \leq 4\|v-\alpha_{t}(p)\|_{2}
    = 4\|v-p\|_{2},
\end{align*}
the second inequality following since $v^{*}\alpha_{t}(x)v \in M$. Hence $\varepsilon_{2t}(Q) \le 4\gamma_t(Q)$. 

The second inequality is immediate, since any $v \in \mathcal{V}_t$ with $vx=\alpha_{t}(x)v$ for all $x \in Q$ also has $v^*\alpha_t(Q)v = Q \subseteq M$. 

For the third inequality, we use Proposition \ref{L:intertwining on big sets}.
Fix $t\in \R,$ and to simplify notation let $\varepsilon=\varepsilon_{t}(Q).$ 
Let $a\in\widetilde{M}$ be the unique element of minimal $\|\cdot\|_{2}$-norm in the closed convex hull of $\{\alpha_{t}(u)u^{*}:u\in \mathcal{U}(Q)\}.$ Then $\|a-p\|_{2}\leq \varepsilon.$ Let $a=v_{1}|a|$ be the polar decomposition of $a,$ and let $q=v_{1}^{*}v_{1}.$ Observe that $q\leq p,$ and that $p|a|=|a|.$ As in the proof of Proposition \ref{L:intertwining on big sets}, we get that $v_{1}^{*}v_{1}\in Q'\cap pMp,$ that $\alpha_{t}(x)v_{1}=v_{1}x$ for all $x\in Q,$ and that $v_{1}v_{1}^{*}=\alpha_{t}(q).$ Since $0\leq |a|\leq 1,$ we also have that
\[\|p-|a|\|_{2}=\|p(1-|a|)\|_{2}\leq \|p(1-|a|^{2})\|_{2}=\|p-a^{*}a\|_{2}\leq 2\|p-a\|_{2}\leq 2\varepsilon. \]
So,
\[\|v_{1}-a\|_{2}=\|v_{1}(p-|a|)\|_{2}\leq 2\varepsilon\]
and hence \[\|v_1 - p\|_2 \le \|v_1-a\|_2 + \|a-p\|_2 \le 3\varepsilon.\]

By Proposition \ref{L:intertwining on big sets}, we may find a partial isometry $v_{2}\in \widetilde{M}$ so that
\[v_{2}^{*}v_{2}=p-q,\, v_{2}v_{2}^{*}=\alpha_{t}(p-q), \mbox{ and } \alpha_{t}(x)v_{2}=v_{2}x\mbox{ for all $x\in Q.$}\]
Then $v=v_{1}+v_{2}$ is partial isometry in $\widetilde{M}$ with $v^{*}v=p,vv^{*}=\alpha_{t}(p)$ and $\alpha_{t}(x)v=vx$ for all $x\in Q.$ Moreover, as $v_{1}(p-q)=0,$ we have that
\[\|v_{2}\|_{2}=\|p-q\|_{2}=\|(p-v_{1})(p-q)\|_{2}\leq \|p-v_{1}\|_{2}\leq 3\varepsilon.\]
So
\[\|p-v\|_{2}\leq \|p-v_{1}\|_{2}+\|v_{2}\|_{2}\leq 6\varepsilon.\]
Thus $\delta_{t}(Q)\leq 6\varepsilon_{t}(Q).$ \qedhere

\end{proof}

We now prove the main result, recovering Theorem \ref{T:concrete estimate intro}. The proof given here via the preceding proposition is a significant simplification of the argument given in an earlier draft of this paper and gives also an improved estimate; this refinement is due to Stefaan Vaes and we thank him for sharing it with us.
\begin{thm}\label{T:preservation on big interesections}
Let $(M,\tau)$ be a semifinite tracial von Neumann algebra. Let $\alpha_{t}\colon\widetilde{M}\to\widetilde{M}$ be an s-malleable deformation of $(M,\tau)$. Fix a projection $p\in M$ with $\tau(p)<\infty.$ Assume that $Q_{1},Q_{2} \le pMp$ are $\alpha$-rigid with $(Q_{1}\cap Q_{2})'\cap p\widetilde{M}p\subseteq M.$ 

Then $Q_{1}\vee Q_{2}$ is $\alpha$-rigid. Further,
\[\|(\alpha_{t}-\id)\big|_{Q_{1}\vee Q_{2}}\|_{\infty,2}\leq 24\|(\alpha_{t/2}-\id)\big|_{Q_{1}}\|_{\infty,2}.\]

\end{thm}

\begin{proof}
Fix $t \in \R$ and adopt the notations of \Cref{P:two sided estimate intertwiner}. We claim that $\gamma_t(Q_1 \vee Q_2) \le \delta_t(Q_1)$. To establish this, taking any $v \in \mathcal{V}_t$ with $vx = \alpha_t(x)v$ for all $x \in Q_1$, we will show that in fact $v^{*}\alpha_{t}(Q_{1}\vee Q_{2})v\subseteq M$.

To do so, take by Proposition \ref{L:intertwining on big sets} some $w \in \mathcal{V}_t$ with $wy = \alpha_{t}(y)w$ for all $y\in Q_{2}$. 
%\[w^{*}w=p,\,\ ww^{*}=\alpha_{t}(p),\mbox{ and } \alpha_{t}(y)w=wy\mbox{ for all $y\in Q_{2}.$}\]
Then $v^{*}w\in (Q_{1}\cap Q_{2})'\cap p\widetilde{M}p\subseteq M$, so that $b=v^{*}w\in \mathcal{U}(pMp)$. Thus $v^{*}\alpha_{t}(Q_{2})v = bQ_{2}b^* \subseteq M$, and $v^{*}\alpha_{t}(Q_{1})v\subseteq M$ by assumption. Since $x\mapsto v^{*}\alpha_{t}(x)v$ is a normal $*$-homomorphism $pMp\to p\widetilde{M}p$, it follows that $v^{*}\alpha_{t}(Q_{1}\vee Q_{2})v\subseteq M$ as desired.  

With the claimed inequality in hand, we simply apply \Cref{P:two sided estimate intertwiner} to get
\begin{align*}
\varepsilon_{2t}(Q_1 \vee Q_2) \le 4\gamma_t(Q_1 \vee Q_2) \le 4\delta_t(Q_1) \le 24\varepsilon_t(Q_1).
\end{align*}\qedhere

\end{proof}

As promised in the introduction, it is straightforward from Theorem \ref{T:preservation on big interesections} to prove Theorem \ref{T:main theorem intro} since the uniform estimate provided in Theorem \ref{T:preservation on big interesections} makes it simple to pass to a inductive limit of subalgebras.

By way of an illustrative example, we give the following result on increasing chains of property (T) subalgebras, the proof of which is now immediate.

\begin{cor}\label{T:PropT intro}
Let $(\alpha,\beta)$ be an s-malleable deformation of tracial von Neumann algebras $M\subseteq \widetilde{M}.$ Let $(Q_{\iota})_{\iota}$ be an increasing net of  property (T) subalgebras of $M.$ Suppose that there is some $\kappa$ so that $Q_{\kappa}'\cap \widetilde{M}\subseteq M.$ Then, $\bigvee_{\iota}Q_{\iota}\in {\rm Rig}(\alpha).$
\end{cor}

The next corollary gives a general stability result for rigid envelopes. It may be viewed as a prototype for proving various permanence and structural properties under stronger assumptions in the sequel.

\begin{cor}\label{C:existence of maximal subalgebras}
Let $(\widetilde{M},\tau)$ be a semifinite tracial von Neumann, and $M\leq \widetilde{M}$ with $\tau\big|_{M}$ semifinite. Suppose we have an s-malleable deformation $\alpha_{t}\colon \widetilde{M}\to \widetilde{M}.$ Fix a projection $p\in M$ with $\tau(p)<\infty.$ Suppose $Q\leq pMp$ has $Q'\cap p\widetilde{M}p\subseteq M$ and that $Q$ is $\alpha$-rigid. Then there is a rigid envelope $P\leq pMp$ of $Q$ with respect to $\alpha.$
Further, $P$ has the  following properties:
\begin{enumerate}[(i)]
\item If $N\leq pMp$ has $(N\cap P)'\cap p\widetilde{M}p\subseteq M$, and if $N$ is $\alpha$-rigid, then $N\subseteq P.$ \label{I:diffuse absorption}
\item Suppose that $\sigma\in \Aut(\widetilde{M}),$ and that $\sigma(pMp)=pMp,$ and
\[\lim_{t\to 0}\|(\alpha_{t}\circ \sigma-\sigma\circ \alpha_{t})\big|_{P}\|_{\infty,2}=0.\]
If $(\sigma(P)\cap P)'\cap p\widetilde{M}p\subseteq M$, then $\sigma(P)=P.$ \label{I:automorphic absorption}
\end{enumerate}

\end{cor}

\begin{proof}
Let $\mathcal{J}$ be the set of all intermediate von Neumann subalgebras $N$ between $Q$ and $pMp$ so that $N$ is $\alpha$-rigid. Order $\mathcal{J}$ by containment. By Theorem \ref{T:preservation on big interesections}, we know that $\mathcal{J}$ is a directed set. Let $P=\bigvee_{N\in \mathcal{J}}N,$ and let $P_{0}=\bigcup_{N\in \mathcal{J}}N.$ Since $
\mathcal{J}$ is directed, we know that  $P_{0}$ is a $*$-subalgebra of $M,$ and that $P=\overline{P_{0}}^{SOT}.$

%Fix $t_{0}>0$ so that $\|(\alpha_{t}-\id)\big|_{Q}\|_{\infty,2}<1/2,$ and a $t\in \R$ with $|t|<2t_{0}.$ 
Let $x\in P_{0}$ with $\|x\|\leq 1,$ and choose $N\in \mathcal{J}$ with $x\in N.$ By Theorem \ref{T:preservation on big interesections} with $Q_{1}=Q,Q_{2}=N,$ we have for any $t\in\mathbb R$ that $\|(\alpha_{t}-\id)\big|_{N}\|_{\infty,2}\leq 24\|(\alpha_{t/2}-\id)\big|_{Q}\|_{\infty,2}.$ Thus, for all $t\in\mathbb R$
\[\|\alpha_{t}(x)-x\|_{2}\leq 24\|(\alpha_{t/2}-\id)\big|_{Q}\|_{\infty,2}.\]
Since this is true for any $x$ in the unit ball of $P_{0},$ the density of the unit ball of $P_{0}$ in the unit ball of $P$ and the normality of $\alpha_{t}$ imply that for all $t\in \mathbb R$
\[\|(\alpha_{t}-\id)\big|_{P}\|_{\infty,2}\leq 24\|(\alpha_{t/2}-\id)\big|_{Q}\|_{\infty,2}.\]
This shows that $\|(\alpha_{t}-\id)\big|_{P}\|_{\infty,2}\to 0$ as $t\to 0.$ So $P\in \mathcal{J},$ and it is clear that $P$ is  the largest element of $\mathcal{J}.$

(\ref{I:diffuse absorption}): This is clear from the maximality of $P$ and Theorem \ref{T:preservation on big interesections}.

(\ref{I:automorphic absorption}): Since $\|(\alpha_{t}\circ \sigma-\sigma\circ \alpha_{t})\big|_{P}\|_{\infty,2}\to_{t\to 0}0,$ we have that
\begin{align*}
\|(\alpha_{t}-\id)\big|_{\sigma(P)}\|_{\infty,2}&=\|(\alpha_{t}\circ \sigma  \circ \id-\sigma\circ \id) \big|_{P}\circ\sigma^{-1}\big|_{\sigma(P)}\|_{\infty,2}\\
&=\|(\alpha_{t}\circ \sigma  \circ \id-\sigma\circ \id)\big|_{P}\|_{\infty,2}\\
&\leq \|(\sigma\circ \alpha_{t}-\alpha_{t}\circ \sigma)\big|_{P}\|_{\infty,2}+\|(\alpha_{t}-\id)\big|_{P}\|_{\infty,2}\to_{t\to 0}0.
\end{align*}
Hence   $\sigma(P)$ is $\alpha$-rigid. Since $(\sigma(P)\cap P)'\cap p\widetilde{M}p \subseteq M,$ it follows part (\ref{I:diffuse absorption}) that $\sigma(P)\subseteq P.$ By a similar argument as above, we see that $\sigma(P)$ is a maximal rigid subalgebra of $pMp.$ Hence we must have that $\sigma(P)=P.$\qedhere

\end{proof}

Having shown the existence of rigid envelopes containing any rigid subalgebra whose relative commutant in $\widetilde{M}$ is contained in $M$, we spend the rest of the paper deriving general results about maximal rigid subalgebras, as well as investigating maximal rigid subalgebras in the special cases (e.g., in the case of deformations of group von Neumann algebras coming from cocycles, as well as the case that $L^{2}(\widetilde{M})\ominus L^{2}(M)$ is a mixing $M$-$M$ bimodule). We close this section with one more general property of maximal rigid subalgebras. We start with the following well known general fact about s-malleable deformations.

\begin{thm}\label{T:amping up to commutants}
Let $(\widetilde{M},\tau)$ be a semifinite tracial von Neumann algebra and that $M\leq \widetilde{M}$ is such that $\tau\big|_{M}$ is still semifinite. Let $\alpha_{t}\colon \widetilde{M}\to \widetilde{M},$ $t\in \R$ be an s-malleable deformation for $M\leq \widetilde{M}.$ Let $p\in \mathcal{P}(M)$ with $\tau(p)<\infty.$ Suppose that $P\leq pMp$ is $\alpha$-rigid and $P'\cap (p\widetilde{M}p)\leq pMp.$ Then $P\vee (P'\cap (pMp))$ is $\alpha$-rigid.
\end{thm}

\begin{proof}
We first claim that $P'\cap pMp$ is $\alpha$-rigid. Given $\varepsilon>0,$ choose a $t_{0}\in (0,\infty)$ so that
\[\|(\alpha_{t}-\id)\big|_{P}\|_{\infty,2}<\varepsilon,\]
for all $t\in [-t_{0},t_{0}].$
Fix $t\in [-t_{0},t_{0}].$
Given $x\in P'\cap pMp,$ and $u\in \mathcal{U}(P),$ we have that
\begin{align*}
\|\alpha_{t}(x)-u\alpha_{t}(x)u^{*}\|_{2}=\|\alpha_{t}(uxu^{*})-u\alpha_{t}(x)u^{*}\|_{2}&=\|\alpha_{t}(u)\alpha_{t}(x)\alpha_{t}(u)^{*}-u\alpha_{t}(x)u^{*}\|_{2}\\
&\leq \|(\alpha_{t}(u)-u)(\alpha_{t}(x)\alpha_{t}(u)^{*})\|_{2}\\
&+\|u\alpha_{t}(x)(\alpha_{t}(u)^{*}-u^{*})\|_{2}\\
&\leq 2\|x\|\|\alpha_{t}(u)-u\|_{2}\\
&\leq 2\|x\|\varepsilon.
\end{align*}
As $P'\cap p\widetilde{M}p=P'\cap pMp,$ we have that $\E_{P'\cap pMp}(\alpha_{t}(x))$ is the element in $\overline{co}^{\|\cdot\|_{2}}\{u\alpha_{t}(x)u^{*}:u\in \mathcal{U}(P)\}$ of minimal $\|\cdot\|_{2}$-norm. So we may write $\E_{P'\cap pMp}(\alpha_{t}(x))$ as a $\|\cdot\|_{2}$-limit of convex combinations of $u\alpha_{t}(x)u^{*}$ with $u\in \mathcal{U}(P).$ So the above estimate shows that
\[\|\alpha_{t}(x)-\E_{P'\cap pMp}(\alpha_{t}(x))\|_{2}\leq 2\|x\|\varepsilon,\]
and thus by transversality,
\[\|\alpha_{2t}(x)-x\|_{2}\leq 2\|\alpha_{t}(x)-\E_{M}(\alpha_{t}(x))\|_{2}\leq 2\|\alpha_{t}(x)-\E_{P'\cap pMp}(\alpha_{t}(x))\|_{2}\leq 4\|x\|\varepsilon.\]
Thus
\[\|(\alpha_{s}-\id)\big|_{P'\cap pMp}\|_{\infty,2}\leq 4\varepsilon\]
for all $s\in[-2t_{0},2t_{0}].$ So $P'\cap pMp$ is $\alpha$-rigid.

We now show that $P\vee (P'\cap pMp)$ is $\alpha$-rigid. Let
\[G=\mathcal{U}(P)\mathcal{U}(P'\cap pMp),\]
and observe that $G$ is a subgroup of $\mathcal{U}(pMp),$ with $W^{*}(G)=P\vee (P'\cap pMp).$
Therefore to show that $P\vee (P'\cap pMp)$ is $\alpha$-rigid it is enough to show that $\alpha_{t}$ converges uniformly on $G.$ Indeed, because of the group structure, we may then use a standard convex hull argument (e.g., as in the proof of Theorem \ref{T:preservation on big interesections} above, see also \cite[Proposition 5.1]{PopaL2Betti}) to find a partial isometry close to $p$ in $\|\cdot\|_2$ implementing $\alpha_t$ on $W^{*}(G)$, thereby showing that $W^{*}(G) = P\vee (P'\cap pMp)$ is $\alpha$-rigid.

So let $\varepsilon>0$ and choose a $t_{0}\in (0,\infty)$ so that
\[\sup_{|t|<t_{0}}\max(\|(\alpha_{t}-\id)\big|_{P}\|_{\infty,2},\|(\alpha_{t}-\id)\big|_{P'\cap pMp}\|_{\infty,2})<\varepsilon.\]
For $|t|<t_{0}$ and $u\in \mathcal{U}(P),v\in \mathcal{U}(P'\cap pMp)$ we have:
\[\|\alpha_{t}(uv)-uv\|_{2}\leq \|\alpha_{t}(u)-u\|_{2}+\|\alpha_{t}(v)-v\|_{2}<2\varepsilon.\]
Thus
\[\sup_{w\in G}\|\alpha_{t}(w)-w\|_{2}\leq 2\varepsilon.\]
Hence, $\alpha_{t}$ converges uniformly on $G$ and thus $P\vee (P'\cap pMp)$ is $\alpha$-rigid. \qedhere

\end{proof}

The experienced reader may note that  typically in deformation rigidity one not only has that rigidity passes to relative commutants, but more generally to the von Neumann algebra generated by the normalizer. However, in order to get good control over normalizers of maximal rigid subalgebras, one needs to assume more about the inclusion $M\leq \widetilde{M},$ e.g., some sort of ``mixingness" of $L^{2}(\widetilde{M})\ominus L^{2}(M)$ (see Section \ref{S:mixingness}). Without this assumption, the only general result along these lines is something like Theorem \ref{T:amping up to commutants}. Regardless, Theorem \ref{T:amping up to commutants} allows us to say that in the general case the relative commutant of a maximal rigid subalgebra is as small as possible.

\begin{cor}\label{C:minimal commutant}
Let $(\widetilde{M},\tau)$ be a semifinite tracial von Neumann, and $M\leq \widetilde{M}$ with $\tau\big|_{M}$ semifinite. Suppose we have an s-malleable deformation $\alpha_{t}\colon \widetilde{M}\to \widetilde{M}.$ Fix a projection $p\in M$ with $\tau(p)<\infty.$ Suppose that $P\leq pMp$ has $P'\cap (p\widetilde{M}p)=P'\cap pMp$ and that $P$ is a maximal rigid subalgebra in $pMp.$ Then $P'\cap pMp=Z(P).$
\end{cor}

\begin{proof}
By Theorem \ref{T:amping up to commutants}, we have that $P\vee (P'\cap pMp)$ is $\alpha$-rigid. By maximality of $P,$ we thus have
\[P\vee (P'\cap pMp)\leq P,\]
so
\[P'\cap pMp\leq P\]
and this clearly implies that $P'\cap pMp=Z(P).$ \qedhere

\end{proof}

Note that it follows from Corollary \ref{C:minimal commutant} that abelian, maximal rigid subalgebras are automatically maximal abelian in $pMp,$ and that factorial maximal rigid subalgebras are automatically irreducible subfactors in $pMp.$ We shall see in Section \ref{S:mixingness} that this can be upgraded in the case that $L^{2}(\widetilde{M})\ominus L^{2}(M)$ is a mixing $M$-$M$ bimodule: maximal rigid subalgebras are automatically singular in $pMp.$

\section{s-malleable deformations from 1-cohomology}\label{S:1-cohomology}

In this section, we give a nice class of examples of maximal rigid subalgebras coming from cocycles with values in orthogonal  representation. We begin by recalling some terminology.

\begin{defn}
Let $G$ be a countable, discrete group and $\mathcal{H}$ a real Hilbert space. The \emph{orthogonal group} of $\mathcal{H},$ denoted $\mathcal{O}(\mathcal{H}),$ is the group of all invertible, real-linear isometries of $\mathcal{H}.$ A homomorphism $\pi\colon G\to \mathcal{O}(\mathcal{H})$ will be called an \emph{orthogonal representation}.
We say that $\pi$ is \emph{mixing} if for $\xi,\eta\in \mathcal{H}$ the map $g\mapsto \ip{\pi(g)\xi,\eta}$ is in $c_{0}(G,\R).$ We say that $\pi$ is \emph{weak mixing} if $0\in \overline{\pi(G)}^{WOT}.$

A \emph{cocycle} for $\pi$ is a map $c\colon G\to\mathcal{H}$ so that
\[c(gh)=\pi(g)c(h)+c(g)\]
for all $g,h\in G.$ It is clear that cocycles form a real vector space under the obvious scaling and additive structure. The real vector space of all cocycles is denoted $Z^{1}(G,\pi).$ We say that $c$ is \emph{inner} if there is a vector $\xi\in \mathcal{H}$ so that $c(g)=(\pi(g)-1)\xi.$ The space of inner cocycles is denoted by $B^{1}(G,\pi).$ Finally, we set $H^{1}(G,\pi)=Z^{1}(G,\pi)/B^{1}(G,\pi).$
\end{defn}

It is well known that a cocycle is inner if and only if it is bounded, see e.g \cite[Proposition 2.2.]{BHV}. Given a cocycle, we can naturally construct a corresponding $s$-malleable deformation.
% called the Gaussian deformation.
It requires an intermediate object, which is the \emph{Gaussian algebra}.

Fix a real Hilbert space $\mathcal{H}.$ Functorially associated to $\mathcal{H}$ is a pair $(A(\mathcal{H}),\tau)$ where $A(\mathcal{H})$ is an abelian von Neumann algebra and $\tau$ is a faithful, normal state on $A(\mathcal{H}).$ This pair is uniquely determined (up to state preserving isomorphism) by the following axioms:
\begin{itemize}
    \item regarding $\mathcal{H}$ as an additive group, and $\mathcal{U}(A(\mathcal{H}))$ as a multiplicative group, there is a homomorphism $\omega\colon \mathcal{H}\to \mathcal{U}(A(\mathcal{H}))$ so that $A(\mathcal{H})=W^{*}(\omega(\mathcal{H})),$
    \item $\tau(\omega(\xi))=\exp(-\|\xi\|^{2})$ for all $\xi\in \mathcal{H}.$
\end{itemize}
Recall that if $(M,\tau_{M}),(N,\tau_{N})$ are two tracial von Neumann algebras, and $D\subseteq N$ is a weak$^{*}$-dense $*$-subalgebra, then any trace-preserving $*$-homomorphism $\theta\colon D\to M$ automatically extends to a trace-preserving embedding $N\hookrightarrow M.$ From this, it is easy to establish that the above axioms uniquely determine the pair $(A(\mathcal{H}),\tau)$ up to isomorphism. For existence, see e.g. \cite[Section 2.1]{PetersonSinclair}. Since $A(\mathcal{H})$ is abelian, we know that $(A(\mathcal{H}),\tau)\cong (L^{\infty}(X,\mu),\int \cdot\, d\mu)$ for some probability space $(X,\mu).$ Thus one often speaks of the ``Gaussian measure space", however this measure space is not canonically associated to a Hilbert space, whereas the Gaussian algebra is.

Now suppose that $\pi\colon G\to \mathcal{O}(\mathcal{H})$ is an orthogonal representation. The axiomatic description of the Gaussian algebra shows that  every $g\in G$ induces a unique trace-preserving $*$-automorphism $\sigma_{g}$ of $A(\mathcal{H})$ by $\sigma_{g}(\omega(\xi))=\omega(\pi(g)\xi).$ Thus we have an action $G\actson^{\sigma}A(\mathcal{H})$ by trace-preserving automorphisms.
So we can form the crossed product $\widetilde{M}=A(\mathcal{H})\rtimes_{\sigma}G.$ To fix notation, we use $(u_{g})_{g\in G}$ for unitaries in $\widetilde{M}$ which implement the action of $G.$ We now define a deformation which ``exponentiates" any cocycle to an s-malleable deformation.

\begin{defn}\label{D:s-mall def cocycle}
Let $G$ be a countable, discrete group,  $\pi\colon G\to \mathcal{O}(\mathcal{H})$ an orthogonal representation of $G,$ and $c\in Z^{1}(G,\pi).$ Let $\widetilde{M}=A(\mathcal{H})\rtimes_{\sigma}G$ and view $M=L(G)$ as a subalgebra of $\widetilde{M}$ in the obvious way. For $t\in \R,$ define
\[\alpha_{c, t}(au_{g})=a\omega(tc(g))u_{g}\mbox{for all $g\in G,$ $a\in A(\mathcal{H}).$}\]
It is straightforward to check that $\alpha_{c, t}$ extends to a $1$-parameter group of automorphisms of $\widetilde{M}.$ Define $\beta\colon \widetilde{M}\to \widetilde{M}$ by $\beta(\omega(\xi)u_{g})=\omega(-\xi)u_{g}.$ The pair $(\alpha_{c, t},\beta)$ is an s-malleable deformation of $M\leq \widetilde{M}$.
% which we call the \emph{Gaussian deformation}.
\end{defn}
%The Gaussian deformation was first defined in \cite{SinclairGaussian}.
This deformation is due to the last named author \cite[Section 3]{SinclairGaussian}, and was also used to great effect by the last named author and Peterson \cite[Section 3.3]{PetersonSinclair}, as well as in \cite{Va10b, AdrianL2Betti, PopaVaesFree} (see also \cite{ChifanSinclair, PopaVaesHyp} for related techniques).
For some insight into how this deformation interacts with the cohomology of $G$ with values of $\mathcal{H},$ observe that by direct computation we have
\[\|\alpha_{c, t}(u_{g})-u_{g}\|_{2}^{2}=2(1-\exp(-t^{2}\|c(g)\|^{2})).\]
Thus $\alpha_{c, t}$ converges uniformly on $G$ as $t\to 0$ if and only if $c$ is bounded, which is equivalent to $c$ being inner. Since  $G\leq \mathcal{U}(L(G))$ and  $W^{*}(G)=L(G),$ the same convexity arguments as in \cite[Proposition 5.1]{PopaL2Betti},\cite[Theorem 4.5]{PetersonL2} show that in fact $L(G)$ is $\alpha_{c, t}$-rigid if and only if $c$ is inner. We proceed to show that maximal subgroups on which the cocycle is inner naturally give rise to maximal rigid subalgebras. We start with the simple case where $c$ is zero on an infinite subgroup.

\begin{prop}\label{P:kernel pineapples}
Let $G$ be a countable, discrete group, and $\pi\colon G\to \mathcal{O}(\mathcal{H})$ an orthogonal representation on a real Hilbert space $\mathcal{H}$. Let $c\colon G\to \mathcal{H}$ be a cocycle, and let
\[H=\{g\in G:c(g)=0\}.\]
If $H$ is infinite, and if $\pi\big|_{H}$ is weak mixing, then $L(H)$ is a maximal rigid subalgebra for $\alpha_{c},$ the $s$-malleable deformation corresponding to $c$ as defined in Definition \ref{D:s-mall def cocycle}.

\end{prop}

\begin{proof}
% Let $\alpha_{t}$ be the corresponding Gaussian deformation, then
First note that
$\alpha_{c,t}\big|_{L(H)}$ is the identity. Let us argue that $L(H)'\cap \widetilde{M}\subseteq M.$ Equivalently, we have to show that the conjugation action of $H$ on $L^{2}(\widetilde{M})\ominus L^{2}(M)$ has no nonzero fixed vector. It is easy to see that the conjugation action on $L^{2}(\widetilde{M})\ominus L^{2}(M)$ is isomorphic to $\kappa\otimes \lambda_{C}$ where:
\begin{itemize}
    \item $\kappa$ is the Koopman representation of $H$ on $L^{2}(A(\mathcal{H}))\ominus \C1,$
    \item $\lambda_{C}\colon H\to \mathcal{U}(\ell^{2}(G))$ is given by $(\lambda_{C}(h)\xi)(g)=\xi(h^{-1}gh),$ for all $g\in G,h\in H.$
\end{itemize}
Since $\pi\big|_{H}$ is weakly mixing, we know by \cite[Proposition 2.7]{PetersonSinclair} that $\kappa$ is weakly mixing. Hence $\kappa\otimes \lambda_{C}$ has no nonzero fixed vectors, so $L(H)'\cap \widetilde{M}\subseteq M.$

Let $P\leq L(G)$ be the rigid envelope of $L(H)$ with respect to $\alpha_{c,t}.$ By Theorem \ref{T:preservation on big interesections}, we have that
\[\|(\alpha_{c,t}-\id)\big|_{P}\|_{\infty,2}\leq 24 \|(\alpha_{c,t}-\id)\big|_{L(H)}\|_{\infty,2}=0\]
for all small $t.$ Since $t\mapsto \alpha_{c,t}$ is a homomorphism, we have that $\alpha_{c,t}\big|_{P}$ is the identity for all $t.$ But then it is easy to see that $P\leq L(H),$ and so $P=L(H).$ \qedhere

\end{proof}

The proceeding proposition can be bootstrapped to a nicer result, saying that maximal subgroups on which a cocycle is inner produce maximal rigid subalgebras.

\begin{cor}\label{C:max rigid subgroup}
Let $G$ be a countable, discrete group, and $\pi\colon G\to \mathcal{O}(\mathcal{H})$ an orthogonal representation on a real Hilbert space $\mathcal{H}$. Let $c\colon G\to \mathcal{H}$ be a cocycle.
Suppose that $H\leq G$ is an infinite subgroup, that $c\big|_{H}$ is inner, and that $\pi\big|_{H}$ is weak mixing. Let $P$ be the rigid envelope of $L(H).$ Then $P=L(K)$ where $K$ is the unique intermediate subgroup of $G$ containing $H$ which is maximal with respect to the property that $c\big|_{K}$ is inner (such a $K$ exists because $\pi\big|_{H}$ is weak mixing).
\end{cor}

\begin{proof}
Find a vector $\xi\in \mathcal{H}$ so that $c(h)=(\pi(h)-1)\xi$ for all $h\in H.$ Since $\pi\big|_{H}$ is weak mixing, we know that $\pi$ has no nonzero $H$-invariant vectors. Therefore the fact that $c\big|_{K}$ is inner implies that $c(k)=(\pi(k)-1)\xi$ for all $k\in K.$ By maximality of $K,$ we must in fact have that $K=\{g\in G:c(g)=(\pi(g)-1)\xi\}.$

Define $\widetilde{c}(g)=c(g)-(\pi(g)-1)\xi,$ so
\begin{equation}\label{E:bounded distance cocycles}
\sup_{g\in G}\|c(g)-\widetilde{c}(g)\|<\infty.
\end{equation}

Let $\alpha_{c, t},\alpha_{\widetilde{c}, t}$ be the $s$-malleable deformations corresponding to $c,\widetilde{c}.$ Observe that $\alpha_{c, t},\alpha_{\widetilde{c}, t}$ are valued in the same von Neumann algebra (since $c,\widetilde{c}$ are valued in the same Hilbert space).  The estimate (\ref{E:bounded distance cocycles}) is easily seen to imply that
\[\lim_{t\to 0}\|(\alpha_{c, t}-\alpha_{\widetilde{c}, t})\big|_{M}\|_{L^{2}(M)\to L^{2}(\widetilde{M})}= 0.\]
A fortiori,
\[\lim_{t\to 0}\|(\alpha_{c, t}-\alpha_{\widetilde{c}, t})\big|_{M}\|_{\infty,2}= 0.\]
So  a diffuse $P\leq M$ is maximal rigid for $\alpha_{c, t}$ if and only if it is maximal rigid for $\alpha_{\widetilde{c}, t}.$
In particular, Proposition \ref{P:kernel pineapples} implies that $L(K)$ is maximal rigid for $\alpha_{\widetilde{c}, t}$ and thus for $\alpha_{c, t}$ as well.

\end{proof}

\section{Permanence properties of maximal rigid subalgebras}\label{S:permanence}

Having established the existence of maximal rigid subalgebras in Section \ref{S:mainproof}, and also investigated the concrete case of $s$-malleable deformations associated to $1$-cocycles in Section \ref{S:1-cohomology}, we now turn to establishing general permanence properties of maximal rigid subalgebras. For the convenience of the reader, we state the main conclusions here.

\begin{thm}\label{T:main properties of max rig intro}
Suppose $(\alpha,\beta)$ is an s-malleable deformation of tracial von Neumann algebras $M\leq \widetilde{M}.$
\begin{enumerate}
\item \label{I:not mixing assumption intro} Fix $P\in {\rm MaxRig}(\alpha)$ with $P'\cap \widetilde{M}\subseteq M.$
\begin{enumerate}
\item We have that $P'\cap M=Z(P),$\label{I:commutant is center}
 \item Suppose that $(\alpha^{0},\beta^{0})$ is another s-malleable deformation of tracial von Neumann algebras $M^{0}\subseteq \widetilde{M}^{0}.$ If $P^{0}\in {\rm MaxRig}(\alpha^{0})$ and $(P^{0})'\cap \widetilde{M}^{0}\subseteq M^{0},$ then $P\overline{\otimes}P^{0}\in {\rm MaxRig}(\alpha\otimes \alpha^{0}).$ \label{I:Pinsker Product formula intro}
 \item If $Q\in {\rm MaxRig}(\alpha)$  and $(P\cap Q)'\cap \widetilde{M}\subseteq M$, then $P=Q.$
 \item If $P\leq N\leq M,$ and $N\ne P,$ then the Jones index of $P$ inside of $N$ is infinite.
 \item If $P,M$ are factors, and $s\in (0,\infty)$ and we regard $P^{s}$ as a subalgebra of $M^{s},$ then $P^{s}\in {\rm MaxRig}(\alpha\otimes\id_{M_{n}(\C)})$ for any $n>s.$ \label{I:amplifications}
 \end{enumerate}
 \item \label{I:mixing assumption intro} Assume that $L^{2}(\widetilde{M})\ominus L^{2}(M)$ is a mixing $M$-$M$ bimodule, and that $P\in {\rm MaxRig}(\alpha)$ is diffuse.
 \begin{enumerate}
\item We have that $\mathcal{N}_{M}(P)=\mathcal{U}(P).$ \label{I:intro singular}
\item  If $Q\in {\rm MaxRig}(\alpha)$ and a corner of $P$ interwines into $Q$ inside of $M$ (in the sense of Popa), then there are nonzero projections $e\in P,f\in Q$ and a unitary $u\in M$ so that $u^{*}(ePe)u=fQf.$
 \item If $Q\in {\rm MaxRig}(\alpha)$ is diffuse and no nonzero corner of $Q$ is conjugate to a nonzero corner of $P,$ then with $wI_{M}(Q,P)$ the weak intertwining space defined by Popa (\cite[Section 2.6]{PopaWeakInter}, \cite{PopaMC}) we have $wI_{M}(Q,P)=\{0\}.$
 \item For any $\sigma\in \Aut(\widetilde{M},\tau)$ with $\sigma(M)=M$  so that $\sigma(P)\cap P$ is diffuse, and so that $\|\alpha_{t}\circ \sigma -\sigma\circ \alpha_{t}\|_{\infty,2}\to_{t\to 0} 0$ we have that $\sigma(P)=P.$ \label{I:aut invariance intro}
 \end{enumerate}

 \end{enumerate}

 \end{thm}

 As the reader can observe in the statement of Theorem \ref{T:main properties of max rig intro}, there is substantially more we can say in the case that $L^{2}(\widetilde{M})\ominus L^{2}(M)$ is a mixing $M$-$M$ bimodule. Because of this, we defer the investigation of the mixing situation to Section \ref{S:mixingness}. The rest of the Section will be divided into two subsections: one on the behavior of maximal rigid subalgebras under tensor products, and one on the behavior maximal rigid subalgebras under compressions/amplifications. These two subsections together with the work in Section \ref{S:mixingness} provide a proof of Theorem \ref{T:main properties of max rig intro}.

 \begin{rmk}
Before continuing further, it may be good for the reader to note that many of the results in this section are inspired by an informal analogy between elements of ${\rm MaxRig}(\alpha)$ and Pinsker algebras of measure-preserving dynamical systems.
 For instance, it is known (see \cite{GlasThoWe}) that if we consider two probability measure-preserving actions $G\actson (X_{j},\mu_{j}),j=1,2,$ of an amenable group, then the Pinsker algebra of $G\actson (X_{1}\times X_{2},\mu_{1}\otimes \mu_{2})$ is the product of the Pinsker algebras of $G\actson (X_{1},\mu_{1})$ and $G\actson (X_{2},\mu_{2})$. This has been extended to the case where $G$ is sofic under mild assumptions on the action by the second named author (see \cite{Hayes13}). This result on Pinsker algebras is in analogy with Theorem \ref{T:main properties of max rig intro} (\ref{I:Pinsker Product formula intro}).

 \end{rmk}

\subsection{Tensor products of maximal rigid subalgebras}\label{S:tensor those pineapples}

The goal of this subsection is to prove that, under mild hypotheses, the tensor product of two maximal rigid subalgebras of $M$ is itself maximal rigid. Recall that maximal rigid subalgebras of s-malleable deformations are in many ways analogous to Pinsker algebras in the context of entropy of probability measure-preserving actions, as well as Pinsker algebras in the context of $1$-bounded entropy. Because of this, the proofs in this section follow the method of proof first given in \cite{GlasThoWe} (using some of the modifications given in \cite{Hayes13} in the sofic context).
As in \cite{GlasThoWe, Hayes13}, we start by first handling the case of tensoring a maximal rigid subalgebra with the trivial deformation, and then reducing the general case to this.

The idea is to prove a ``tensor splitting lemma", showing that an algebra decomposes as a tensor product if it is invariant under a sufficiently nontrivial trace-preserving action of a group. In \cite{GlasThoWe, Hayes13} this was done when the action was an \emph{ergodic} measure-preserving action, but we still state the analogous version for general trace-preserving actions with the necessary modifications for dealing with the ``space of centrally ergodic components" (i.e. the elements in the center fixed by the acting group). The following is the tensor splitting lemma we need.

\begin{lem}\label{L:tensor splitting}
Let $(M,\tau_{M}),(N,\tau_{N})$ be two tracial von Neumann algebras, and let $G\cc^{\alpha}(N,\tau_{N})$ be a trace-preserving action of a discrete group $G.$ Suppose that $N\otimes 1\leq Q\leq N\overline{\otimes}M$ and $(\alpha_{g}\otimes \id)(Q)=Q$ for all $g\in G.$  Suppose further that there is a $P\leq M$ so that $Q\cap (Z(N)^{G}\overline{\otimes}M)=Z(N)^{G}\overline{\otimes}P.$ Then $Q=N\overline{\otimes}P.$

\end{lem}

\begin{proof}
From  the assumptions, it is clear  that $ N\overline{\otimes}P \subseteq Q$, and hence it suffices to show $Q\subseteq N\overline{\otimes} P$. 
%Indeed if $ y\in N\overline{\otimes} P\setminus Q $, then $\E_{Z(N)^G\otimes M}(y)\not\in Q$ violating our hypotheses.   

For each $g\in G,$ we use $\alpha_{g}$ for the unique continuous extension of $\alpha_{g}$ to a unitary $L^{2}(N)\to L^{2}(N).$
Consider the conditional expectation onto $Q$ as a projection operator $\E_{Q}\in B(L^{2}(N)\otimes L^{2}(M)).$ We then have by Tomita's theorem that
\begin{align*}
    \E_{Q}\in [(N\cup JNJ\cup \{\alpha_{g}:g\in G\})\otimes 1]'&=[Z(N)\cap \alpha(G)']\overline{\otimes}B(L^{2}(M))\\
    &=Z(N)^{G}\overline{\otimes}B(L^{2}(M)).
\end{align*}
Observe that this implies that $\E_{Q}(1\otimes a)\in  Z(N)^{G}\overline{\otimes} M$ for all $a\in M.$ We thus have that
\[\E_{Q}(1\otimes a)\in [Z(N)^{G}\overline{\otimes} M]\cap Q=Z(N)^{G}\overline{\otimes}P\]
for all $a\in M.$ 
%We now make the following claim.
This now implies  that $Q=\E_Q(N\overline{\otimes} M)\subseteq N\overline{\otimes} P$. \qedhere
%
%\emph{Claim: For every $a,b\in M$ we have that $\ip{\E_{Q}(1\otimes a),1\otimes  b}=\ip{\E_{P}(a),b}.$}
%
%To prove this, write $a=\E_{P}(a)+a_{0},b=\E_{P}(b)+b_{0}.$ Now let us compute:
%\begin{align*}
%\ip{\E_{Q}(1\otimes a),1\otimes  b}&=\ip{\E_{P}(a),\E_{P}(b)}+\ip{\E_{Q}(1\otimes a_{0}),1\otimes \E_{P}(b)}\\
%&+\ip{\E_{Q}(1\otimes \E_{P}(a)),1\otimes b_{0}}+\ip{\E_{Q}(1\otimes a_{0}),1\otimes b_{0}}\\
%&=\ip{\E_{P}(a),b}+\ip{1\otimes a_{0},\E_{Q}(1\otimes \E_{P}(b))}\\
%&+\ip{1\otimes \E_{P}(a),1\otimes b_{0}}+\ip{\E_{Q}(1\otimes a_{0}),1\otimes b_{0}}\\
%&=\ip{\E_{P}(a),b}+\ip{\E_{Q}(1\otimes a_{0}),1\otimes b_{0}}.
%\end{align*}
%Where in the second and third equalities we use  that $1\otimes P\leq Q.$ Since $\E_{Q}(1\otimes a_{0})\in Z(N)^{G}\overline{\otimes}P,$ we have that $\ip{\E_{Q}(1\otimes a_{0}),1\otimes b_{0}}=0.$ This proves  the claim.
%
%
%Now fix an $a\in M.$ Then by the claim:
%\[\|\E_{Q}(1\otimes a)\|_{2}^{2}=\ip{\E_{Q}(1\otimes a),1\otimes a}=\ip{\E_{P}(a),a}=\|\E_{P}(a)\|_{2}^{2},\]
%\[\ip{\E_{Q}(1\otimes a),1\otimes \E_{P}(a)}=\|\E_{P}(a)\|_{2}^{2},\]
%so it follows from the Cauchy-Schwartz inequality that $\E_{Q}(1\otimes a)=1\otimes \E_{P}(a).$ Since $N\otimes 1\leq Q,$ we have for all $b\in N,a\in M$ that
%\[\E_{Q}(b\otimes a)=(b\otimes 1)\E_{Q}(1\otimes a)=b\otimes \E_{P}(a).\]
%Hence $\E_{Q}=\id\otimes \E_{P},$ and this shows that $Q=N\overline{\otimes}P.$

\end{proof}

From the preceding lemma, we can easily handle the case of tensoring with a trivial deformation.

\begin{cor}\label{C:tensor permanence 1}
Let  $(M,\tau_{M})$ be a semifinite tracial von Neumann algebra, and $\alpha_{t}\colon \widetilde{M}\to \widetilde{M}$ an s-malleable deformation. Fix a projection $p\in M$ with $\tau(p)<\infty.$ Assume that $P\leq pMp$ has $P'\cap p\widetilde{M}p\subseteq pMp$ and that $P\leq pMp$ is a maximal rigid subalgebra for $\alpha_{t}.$ Fix a tracial von Neumann algebra $(N,\tau_{N}).$ Then $N\overline{\otimes}P$, regarded as a subalgebra of $N\overline{\otimes}pMp,$ is a maximal rigid subalgebra for $\id\otimes\alpha_{t}.$

\end{cor}

\begin{proof}
Let $Q\leq N\overline{\otimes}pMp$ be the rigid envelope of $N\overline{\otimes}P$ with respect to $\id\otimes \alpha_{t}.$  This is possible by Corollary \ref{C:existence of maximal subalgebras} and Tomita's commutation theorem. We prove that $Q=N\overline{\otimes}P$ in two steps.

\emph{Step 1: We show that $Q\cap (Z(N)\overline{\otimes}pMp)=Z(N)\overline{\otimes}P$}

To prove this, let $\widehat{Q}\leq Z(N)\overline{\otimes}pMp$ be the rigid envelope of  $Q\cap (Z(N)\overline{\otimes}pMp).$ It is possible to do this by Corollary \ref{C:existence of maximal subalgebras}. Let $G=\Aut(Z(N),\tau),$ and let $G\actson Z(N)$ in the canonical way. Since $Z(N)$ is abelian, this action is clearly ergodic. Then
\[Z(N)\overline{\otimes}P\leq \widehat{Q}\leq Z(N)\overline{\otimes}pMp,\]
and $(\Theta\otimes \id)(\widehat{Q})=\widehat{Q}$ for all $\Theta\in G,$ by maximality. Hence, by Lemma \ref{L:tensor splitting} we have that $\widehat{Q}=Z(N)\overline{\otimes}\widehat{P},$ where $\widehat{P}$ is such that $Q\cap (1\otimes M)=1\otimes \widehat{P}.$ Then $\widehat{P}$ is $\alpha$-rigid and contains $P$. Thus $\widehat{P}=P$ by maximality of $P$. Since $\widehat{Q}$ contains $Q\cap (Z(N)\overline{\otimes}pMp),$ this proves Step 1.

\emph{Step 2: We prove the corollary.}

Now let $G=\{1\}.$ We then have that
\[N\overline{\otimes}1\leq Q\leq N\overline{\otimes}pMp.\]
By Step 1 and Lemma \ref{L:tensor splitting} we have that $Q=N\overline{\otimes}P,$ so we are done. \qedhere

\end{proof}

As in \cite{GlasThoWe, Hayes13}, we can bootstrap the case of tensoring with the trivial deformation to the general case.

\begin{thm}\label{T:tensor permanence 2}
Let $(M_{j},\tau_{j}),j=1,2$ be two semifinite tracial von Neumann algebras, and $\alpha_{t}^{j}\colon \widetilde{M}_{j}\to \widetilde{M}_{j},j=1,2$ be two s-malleable deformations. Fix $p_{j}\in \mathcal{P}(M_{j})$ with $\tau(p_{j})<\infty$ for $j=1,2.$ For $j=1,2$ let $P_{j}\leq p_{j}M_{j}p_{j}$ be subalgebras so that $P_{j}'\cap p_{j}\widetilde{M}_{j}p_{j}\subseteq p_{j}M_{j}p_{j}.$ Assume moreover that $P_{j}$ is maximal rigid for $\alpha_{t}^{j}$ for $j=1,2.$ Then $P_{1}\overline{\otimes}P_{2},$ regarded as a subalgebra of $(p_{1}\otimes p_{2})(M_{1}\overline{\otimes}M_{2})(p_{1}\otimes p_{2}),$ is maximal rigid for $\alpha_{t}^{1}\otimes \alpha_{t}^{2}.$

\end{thm}

\begin{proof}
Let $Q\leq p_{1}M_{1}p_{1}\overline{\otimes}p_{2}M_{2}p_{2}$ be the rigid envelope of $P_{1}\overline{\otimes}P_{2}$ with respect to $\alpha_{t}^{1}\otimes\alpha_{t}^{2}.$
By symmetry, it is enough to show the following:

\emph{Claim: $Q\subseteq P_{1}\overline{\otimes}p_{2}M_{2}p_{2}$.}

To prove the claim, note that for all $x\in Q$ we have:
\[\|(\E_{M_{1}}\otimes \E_{M_{2}}) (\alpha_{t}^{1}\otimes \alpha_{t}^{2})(x)\|_{2}\leq \|(\E_{M_{1}}\otimes \id)(\alpha_{t}^{1}\otimes \id)(x)\|_{2}.\]
Let $\varepsilon>0.$ We may find a $t_{0}>0$ so that for all $|t|<t_{0},$ we have
\[\|(\alpha_{t}^{1}\otimes\alpha_{t}^{2}-\id)\big|_{Q}\|_{\infty,2}<\varepsilon.\]
Note that this implies
\[\|(\E_{M_{1}}\otimes \E_{M_{2}}) (\alpha_{t}^{1}\otimes \alpha_{t}^{2})(x)\|_{2}\geq \|x\|_{2}-\varepsilon,\]
for all $x\in (Q)_{1},$ and $|t|<t_{0}.$ Hence we have
\[ \|(\E_{M_{1}}\otimes \id) (\alpha_{t}^{1}\otimes \id)(x)\|_{2}\geq\|x\|_{2}-\varepsilon\]
for all $|t|<t_{0}.$ Thus for all $|t|<t_{0},x\in Q$ with $\|x\|_{\infty}\leq 1$ and $\|x\|_{2}\geq \varepsilon$
\begin{align*}
\|(\alpha_{t}^{1}\otimes \id)(x)-(\E_{M_{1}}\otimes \id) (\alpha_{t}^{1}\otimes \id)(x)\|_{2}^{2}&=\|(\alpha_{t}^{1}\otimes \id)(x)\|_{2}^{2}-\|(\E_{M_{1}}\otimes \id) (\alpha_{t}^{1}\otimes \id)(x)\|_{2}^{2}\\
&\leq 2\varepsilon\|x\|_{2}-\varepsilon^{2}\\
&\leq 2\varepsilon\sqrt{\tau_{1}(p_{1})\tau_{2}(p_{2})}.
\end{align*}
Whereas if $x\in Q,\|x\|_{\infty}\leq 1$ and $\|x\|_{2}<\varepsilon$ we have that
\[\|(\alpha_{t}^{1}\otimes \id)(x)-(\E_{M_{1}}\otimes \id)(\alpha_{t}^{1}\otimes \id)(x)\|_{2}<2\varepsilon.\]
Hence we see that
\[\|(\alpha_{t}^{1}\otimes \id-\E_{M_{1}}\otimes \id)\big|_{Q}\|_{\infty,2}\leq \max(2\varepsilon,\sqrt{2\varepsilon}(\tau(p_{1})\tau(p_{2}))^{1/4}).\]
We now apply the transversality estimate for $\alpha_{t}^{1}\otimes \id:\widetilde{M}_{1}\overline{\otimes}M_{2}\to \widetilde{M}_{1}\overline{\otimes}M_{2}$ to see that $\alpha_{t}^{1}\otimes\id_{M_{2}}$ converges uniformly on $(Q)_{1}.$

Observe that
\[Q'\cap (p_{1}\widetilde{M}_{1}p_{1}\overline{\otimes}p_{2}\widetilde{M}_{2}p_{2})\subseteq (P_{1}\overline{\otimes}P_{2})'\cap (p_{1}\widetilde{M}_{1}p_{1}\overline{\otimes}p_{2}\widetilde{M}_{2}p_{2})\subseteq M_{1}\overline{\otimes}M_{2}.\]
Moreover, since $Q$ and $P_{1}\overline{\otimes}M_{2}$ both contain $P_{1}\overline{\otimes}P_{2},$
it follows by Corollary \ref{C:tensor permanence 1} and Corollary \ref{C:existence of maximal subalgebras} (\ref{I:diffuse absorption}) that $Q\subseteq P_{1}\overline{\otimes}M_{2}.$ \qedhere

\end{proof}

\subsection{Compressions of maximal rigid subalgebras}\label{S:compressions/amplifications}

In this subsection we explain how the property of being maximal rigid behaves under amplifications/compressions. Very roughly, what we will see is that if $N\leq M$ are tracial von Neumann algebras, and $(\widetilde{M},\alpha_{t})$ is an s-malleable deformation of $M,$ then $N^{s}\leq M^{s}$ is maximal rigid ``with respect to $\alpha$" if and only if $N\leq M$ is maximal rigid with respect to $\alpha.$ Of course, this result is not true as stated because $\alpha$ is not defined on $M^{s},$ and if $s$ is not integer, then there is no clear definition for $\alpha^{s}$ as a deformation of $M^{s}.$ Because of this, we will work in the situation that $M$ is semifinite, and $N$ is subalgebra of a finite trace corner. In the above discussion, this amounts to replacing $M$ with $M\overline{\otimes}B(\ell^{2}(\N)),$ and $\widetilde{M}$ with $\widetilde{M}\overline{\otimes}B(\ell^{2}(\N)).$ The advantage of this situation is that it makes it clear what the analogue of $N^{s}\leq M^{s}$ is. We are simply taking corners of $N,$ and either moving between ones of larger or smaller trace.
We start with the following proposition.

\begin{prop}\label{P:amplifying up from corners}
Let $(M,\tau)$ be a semifinite tracial von Neumann algebra, and $\alpha_{t}\colon \widetilde{M}\to \widetilde{M}$ an s-malleable deformation of $M.$ Let $e\in \mathcal{P}(M)$ Let $P\subseteq eMe$ be a von Neumann subalgebra of $eMe,$ and $p$ a projection in $P$ with $\tau(p)<\infty.$ Let $f$ be the central support of $p$ in $P.$ If $pPp$ is $\alpha$-rigid, then $qPq$ is $\alpha$-rigid for any projection $q\in P$ with $q\leq f$ and $\tau(q)<\infty.$
\end{prop}

\begin{proof}

We may choose a family $\{v_{i}\}_{i\in I}$ of partial isometries in $P$ so that $v_{i}^{*}v_{i}\leq p,$ and $\sum_{i}v_{i}v_{i}^{*}=q.$ For $F\subseteq I$ finite, let $p_{F}=\sum_{i\in F}v_{i}v_{i}^{*}.$ It is straightforward to show that  $p_{F}Pp_{F}$ is $\alpha$-rigid for all $F\subseteq I$ finite.

Fix $\varepsilon>0,$ and choose a finite $F\subseteq I$ so that $\|p_{F}-q\|_{2}<\varepsilon.$ For $x\in qPq,$ we then have:
\begin{align*}
\|\alpha_{t}(x)-x\|_{2}&\leq \|\alpha_{t}(qxq)-\alpha_{t}(p_{F}xp_{F})\|_{2}+\|qxq-p_{F}xp_{F}\|_{2}+\|\alpha_{t}(p_{F}xp_{F})-p_{F}xp_{F}\|_{2}\\
&\leq 4\varepsilon\|x\|_{\infty}+\|\alpha_{t}(p_{F}xp_{F})-p_{F}xp_{F}\|_{2}.
\end{align*}
Hence,
\[\|(\alpha_{t}-\id)\big|_{qPq}\|_{\infty,2}\leq 4\varepsilon+\|(\alpha_{t}-\id)\big|_{p_{F}xp_{F}}\|_{\infty,2}.\]
Letting $t\to 0$ we see that
\[\limsup_{t\to 0}\|(\alpha_{t}-\id)\big|_{qPq}\|_{\infty,2}\leq 4\varepsilon,\]
and letting $\varepsilon\to 0$ completes the proof. \qedhere

\end{proof}

From Proposition \ref{P:amplifying up from corners}, it is simple to establish the permanence results we want for (generalized) amplifications/compressions of maximal rigid subalgebras.

\begin{cor}\label{C:corners of pineapples}
Let $(M,\tau)$ be a semifinite tracial von Neumann algebra, and $\alpha_{t}\colon \widetilde{M}\to \widetilde{M}$ an s-malleable deformation of $M.$
\begin{enumerate}[(a)]
\item Suppose that $p\in \mathcal{P}(M)$ with $\tau(p)<\infty,$ and let $P\leq pMp$ be  maximal rigid  with respect to $\alpha_{t}.$ Then for any projection $q$ in $P$ we have that $qPq$ (regarded as subalgebra of $pMp$) is a maximal rigid subalgebra with respect to $\alpha_{t}.$ \label{I:compressing down}
\item Suppose that $e,p\in \mathcal{P}(M)$ with $\tau(p)<\infty,$ and $p\leq e.$ Suppose that $P\leq eMe$ and that $pPp$ regarded as a subalgebra of $pMp$ is a maximal rigid subalgebra with respect to $\alpha_{t}.$ Let $f$ be the central support of $p$ in $P.$ Then for every $q\leq f$ with $\tau(q)<\infty,$ we have that $qPq$ is a maximal rigid subalgebra in $qMq$ with respect to $\alpha_{t}.$ \label{I:amplifying up}
\end{enumerate}
\end{cor}

\begin{proof}

(\ref{I:compressing down}): Suppose that $qPq\leq Q\leq qMq$ and that $\alpha_{t}$ converges uniformly on the unit ball of $Q.$ Let $e$ be the central support of $q$ in $P,$ and let $\{v_{i}\}_{i\in I}$ be partial isometries in $P$ so that $v_{i}^{*}v_{i}\leq q,$ and $\sum_{i}v_{i}v_{i}^{*}=e.$ We may, and will, assume moreover that there is an $i_{0}\in I$ so that $v_{i_{0}}=q.$ Let $\widehat{Q}=\{x\in eMe:v_{i}^{*}xv_{j}\in Q\mbox{ for all $i,j\in I$}\}.$ It is easy to establish that $\widehat{Q}$ is a subalgebra of $eMe,$ and that $q\widehat{Q}q=Q.$ By Proposition \ref{P:amplifying up from corners}, we know that $\widehat{Q}$ is $\alpha$-rigid. Because $e$ is central in $P,$ we know that $P(p-e)+\widehat{Q}$ is a von Neumann subalgebra of $pMp,$ and it is easy to see that  $P(p-e)+\widehat{Q}$ is $\alpha$-rigid. Since $P$ is maximal rigid with respect to $\alpha_{t},$ and $P\subseteq P(p-e)+\widehat{Q},$ it follows that $P=P(p-e)+\widehat{Q}.$ Compressing by $q,$ we have that $qPq=q\widehat{Q}q=Q.$

(\ref{I:amplifying up}): By (\ref{I:compressing down}), it is enough to show that $(p\vee q)P(p\vee q)$ is  maximal rigid in $(p\vee q)M(p\vee q)$ with respect to $\alpha_{t}.$ By Proposition \ref{P:amplifying up from corners}, we know that  $(p\vee q)P(p\vee q)$ is $\alpha$-rigid. Suppose $Q\leq (p\vee q)M(p\vee q)$ is $\alpha$-rigid and contains $(p\vee q)P(p\vee q).$ It is easy to see that $pQp$ is $\alpha$-rigid in $pMp.$ Since $pPp\leq pQp,$ this implies that $pPp=pQp.$ Since $p$ has central support equal to $p\vee q$ in $(p\vee q)P(p\vee q),$ it is easy to see that the equality $pPp=pQp$ and the inclusions $(p\vee q)P(p\vee q)\leq Q\leq (p\vee q)M(p\vee q)$ imply that $(p\vee q)P(p\vee q)=Q.$ \qedhere

\end{proof}

In many situations, we will start with a tracial von Neumann algebra $(M_{0},\tau_{0})$ and an s-malleable deformation $\alpha_{0,t}\colon \widetilde{M_{0}}\to \widetilde{M_{0}}.$ We may then set $M=M_{0}\overline{\otimes}B(\mathcal{H}),$ $\widetilde{M}=\widetilde{M_{0}}\overline{\otimes}B(\mathcal{H}),$ $\alpha_{t}=\alpha_{0,t}\otimes \id,$ $\tau=\tau_{0}\otimes \Tr$ for some infinite dimensional Hilbert space $\mathcal{H}.$ We then get a particularly strong version of maximal rigid subalgebras being  preserved under amplifications: if $P\leq M_{0}$ is maximal rigid with respect to $\alpha_{0,t}$ and if $p\in \mathcal{P}(M)$ with $\tau(p)<\infty,$ then $p(P\overline{\otimes}B(\mathcal{H}))p$ remains maximal rigid in $pMp$ with respect to $\alpha_{t}.$

One possible way to view Proposition \ref{P:amplifying up from corners} is the following. Suppose $p\in \mathcal{P}(M),$ and that $P\leq pMp$ is maximal rigid  with respect to some s-malleable deformation. We can try to find a maximal rigid $\widehat{P}\leq M$ with $p\widehat{P}p=P.$ The difficulty here is that we do not get a unique such $\widehat{P}.$ However, we can classify all possible choices of $\widehat{P}$ assuming that $P$ and $M$ are $\textrm{II}_{1}$-factors.

\begin{prop}\label{P:lifting pineapples from corners}
Let $M$ be a $\textrm{II}_{1}$-factor and $\alpha_{t}\colon \widetilde{M}\to \widetilde{M}$ an s-malleable deformation of $M.$ Let $p\in \mathcal{P}(M)$ be nonzero and $P\leq pMp$ be maximal rigid and diffuse.
\begin{enumerate}[(i)]
\item There is a maximal rigid, diffuse $\widehat{P}\leq M$ with $p\widehat{P}p=P,$  and so that the central support of $p$ in $\widehat{P}$ is $1.$

\label{I:existence of extensions}
\item If $P$ is a factor, if $\widehat{P}_{j}\leq M,j=1,2$ satisfy $p\widehat{P}_{j}p=P,j=1,2,$ and if the central support of $p$ in $\widehat{P}_{j}$ is $1$ for $j=1,2,$ then there is a $u\in \mathcal{U}(M)\cap \{p\}'$ so that $u\widehat{P}_{1}u^{*}=\widehat{P}_{2}$ and $uPu^{*}=P.$ \label{I:unitary conjugation of extensions}
\end{enumerate}

\end{prop}

\begin{proof}

(\ref{I:existence of extensions}): Since $P$ is diffuse and $M$ is a factor, we may find a collection of nonzero partial isometries $(v_{i})_{i\in I}$ in $M$ so that
\begin{itemize}
    \item $\sum_{i}v_{i}v_{i}^{*}=1,\mbox{ and  } v_{i}^{*}v_{i}\leq p,$
    \item $v_{i}^{*}v_{i}\in pPp,$
    \item there is some $i_{0}\in I$ with $v_{i_{0}}=p.$
\end{itemize}
Let
\[\widehat{P}=\{x\in M:v_{i}^{*}xv_{j}\in P\mbox{ for all $i,j$}\}.\]
We leave it as an exercise to verify to the reader  that this is indeed a von Neumann subalgebra of $M,$ and that $p\widehat{P}p=P.$ Moreover, the central support of $p$ in $\widehat{P}$ is $1$ by construction.  By  Corollary \ref{C:corners of pineapples} (\ref{I:amplifying up}) we know that $\widehat{P}$ is maximal rigid.

(\ref{I:unitary conjugation of extensions}):
Since $P$ is a factor and the central support of $p$ in $\widehat{P}_{j},j=1,2$ is $1$, we know that each $\widehat{P}_{j}$ is a factor. So  for $j=1,2$ we can find a family of partial isometries $(v_{i}^{j})_{i\in I}$ in $\widehat{P}_{j}$ so that
\begin{itemize}
    \item $(v_{i}^{j})^{*}v_{i}^{j}\in P$ for $j=1,2$ and all $i\in I,$
    \item there is an $i_{0}\in I$ with $v_{i_{0}}^{1}=p=v_{i_{0}}^{2},$
    \item $(v_{i}^{1})^{*}v_{i}^{1}=(v_{i}^{2})^{*}v_{i}^{2},$
    \item $\sum_{i\in I}v_{i}^{j}(v_{i}^{j})^{*}=1,$ for $j=1,2.$
\end{itemize}
Set $u=\sum_{i\in I}v_{i}^{2}(v_{i}^{1})^{*},$ it is then easy to see that $u\in \mathcal{U}(M\cap \{p\}'),$  that $u\widehat{P}_{1}u^{*}=\widehat{P}_{2},$ and that $uPu^{*}=P.$ \qedhere

\end{proof}

It is possible to give a nonfactorial version of Proposition \ref{P:lifting pineapples from corners}. However in this situation we need to assume in addition that not only is $P$ diffuse, but that it is ``diffuse over $Z(M)$", (i.e. $P\nprec_M Z(M),$ see Section \ref{S:mixingness} for the precise definitions). The nonfactorial version follows from the same arguments as in Proposition \ref{P:lifting pineapples from corners}, but requires a more delicate usage of the center-valued trace. The central issue is at the beginning of part (\ref{I:existence of extensions}) we have to find a family of partial isometries $(v_{i})_{i\in I}$ so that
\begin{itemize}
    \item $\sum_{i}v_{i}v_{i}^{*}=1,\mbox{ and  } v_{i}^{*}v_{i}\leq p,$
    \item $v_{i}^{*}v_{i}\in pPp,$
    \item there is some $i_{0}\in I$ with $v_{i_{0}}=p.$
\end{itemize}
This is clearly possible when $P$ is diffuse and $M$ is a factor but presents a problem in the general case. One needs to know  that the ``components" of $P$ in the integral decomposition of $M$ over its center are almost everywhere diffuse. This is indeed the case if $P\nprec_M Z(M),$ but proving this is sufficient requires some technical arguments.
In order to simplify the presentation, and since the factorial case is our main concern, we have elected to not provide the proof of this more general situation.

Since it follows from methods similar to \ref{P:amplifying up from corners}, we close with a result showing that rigidity passes up from finite index subalgebras.
Important in the proof is the usage of an orthonormal basis for over algebra over another, as defined by Pimsner-Popa in \cite[Propostion 1.3]{PPEntropy}.

\begin{prop}
Let $(M,\tau)$ be a  semifinite tracial von Neumann algebra and let $\alpha_{t}\colon \widetilde{M}\to \widetilde{M}$ be an s-malleable deformation of $M.$ Let $e\in \mathcal{P}(M)$ with $\tau(e)<\infty.$ Suppose that $N\leq eMe$ is in $\alpha$-rigid and that $Q$ is an intermediate von Neumann subalgebra between $N$ and $eMe.$ If the Jones index of $N$ inside of $Q$ is finite, then $Q$ is $\alpha$-rigid.
\end{prop}

\begin{proof}
Choose a Pimsner-Popa basis $(m_{j})_{j=1}^{\infty}$ for $Q$ over $N.$
 Namely, $m_{j}\in Q$ and satisfy the following:
 \begin{itemize}
     \item $\E_{N}(m_{j}^{*}m_{i})=\delta_{i=j}f_{j}$ for some projections $f_{j}\in \mathcal{P}(N),$
     \item $\xi=\sum_{j}m_{j}\E_{N}(m_{j}^{*}\xi)$ for every $\xi\in L^{2}(Q),$ where the sum converges in $\|\cdot\|_{2}$
 \end{itemize}
 (see e.g. the proof of \cite[Theorem 1.1.6]{PopaCBMS}).
 Then the Jones index of $N$ inside of $Q$ is easily seen to be $\frac{1}{\tau(e)}\sum_{j}\tau(m_{j}^{*}m_{j}),$ so $\sum_{j}\tau(m_{j}m_{j}^{*})<\infty.$

 Let $\varepsilon>0,$ we may choose a $K\in \N$ so that $\sum_{j>K}\tau(m_{j}m_{j}^{*})<\varepsilon^{2}.$ Let $x\in Q,$ and $t\in \R.$ Set $x_{K}=\sum_{j=1}^{K}m_{j}\E_{N}(m_{j}^{*}x).$ We then have that
 \begin{equation}\label{E:same trick as before index}
 \|\alpha_{t}(x)-x\|_{2}\leq \|\alpha_{t}(x_{K})-x_{K}\|_{2}+2\|x-x_{K}\|_{2}.
 \end{equation}
 We start by estimating $\|x-x_{K}\|_{2}.$ We have that
 \begin{equation}\label{E:small pile of garbage}
     \|x-x_{K}\|_{2}^{2}=\sum_{j>K}\|m_{j}\E_{N}(m_{j}^{*}x)\|_{2}^{2}.
 \end{equation}
 For a $j>K,$
 \[\|m_{j}\E_{N}(m_{j}^{*}x)\|_{2}^{2}=\tau(\E_{N}(x^{*}m_{j})m_{j}^{*}m_{j}\E_{N}(m_{j}^{*}x))=\tau(m_{j}\E_{N}(m_{j}^{*}x)\E_{N}(x^{*}m_{j})m_{j}^{*}).\]
 Consider the Jones basic construction $\ip{Q,e_{N}},$ where $e_{N}$ is the Jones projection, and let  $\Tr$ be the natural trace of $\ip{Q,e_{N}}$ which satisfies $\Tr(ae_{N}b)=\frac{1}{\tau(e)}\tau(ab)$ for all $a,b\in Q.$
 The above then shows that
 \begin{align*}
  \|m_{j}\E_{N}(m_{j}^{*}x)\|_{2}^{2}=\tau(e)\Tr(m_{j}e_{N}m_{j}^{*}xe_{N}x^{*}m_{j}e_{N}m_{j}^{*})&\leq \tau(e)\|x\|_{\infty}^{2}\Tr(m_{j}e_{N}m_{j}^{*}m_{j}e_{N}m_{j}^{*})\\
  &=\tau(e)\|x\|_{\infty}^{2}\Tr(m_{j}\E_{N}(m_{j}^{*}m_{j})e_{N}m_{j}^{*})\\
  &=\|x\|_{\infty}^{2}\tau(m_{j}f_{j}m_{j}^{*})\\
  &\leq \|x\|_{\infty}^{2}\tau(m_{j}m_{j}^{*}).
 \end{align*}
Inserting this into (\ref{E:small pile of garbage}) we have  $\|x-x_{K}\|_{2}^{2}\leq \|x\|_{\infty}^{2}\sum_{j>K}\tau(m_{j}m_{j}^{*})<\varepsilon^{2}\|x\|_{\infty}^{2}.$ Applying this estimate with $(\ref{E:same trick as before index}),$
\[\|\alpha_{t}(x)-x\|_{2}\leq 2\varepsilon\|x\|_{\infty}+\|\alpha_{t}(x_{K})-x_{K}\|_{2}.\]
Letting $C_{K}=\max_{1\leq j\leq K}\|m_{j}\|_{\infty},$ it is easy to see that
\begin{align*}\|\alpha_{t}(x_{K})-x_{K}\|_{2}&\leq KC_{K}\max_{1\leq j\leq K}\left(\|\alpha_{t}(m_{j})-m_{j}\|_{2}\|x\|_{\infty}+\|\alpha_{t}(\E_{N}(m_{j}^{*}x))-\E_{N}(m_{j}^{*}x)\|_{2}\right)\\
&\leq KC_{K}\|x\|_{\infty}\max_{1\leq j\leq K}\left(\|\alpha_{t}(m_{j})-m_{j}\|_{2}+C_{K}\|(\alpha_{t}-\id)\big|_{N}\|_{\infty,2}.\right)
\end{align*}
So
\[\|(\alpha_{t}-\id)\big|_{Q}\|_{\infty,2}\leq 2\varepsilon+C_{K}^{2}K\|(\alpha_{t}-\id)\big|_{N}\|_{\infty,2}+ KC_{K}\max_{1\leq j\leq K}\|\alpha_{t}(m_{j})-m_{j}\|_{2}.\]
So
\[\limsup_{t\to 0}\|(\alpha_{t}-\id)\big|_{Q}\|_{\infty,2}\leq 2\varepsilon.\]
As this is true for every $\varepsilon>0,$ we can let $\varepsilon\to 0$ to complete the proof.
\end{proof}

\section{Consequences in the Mixing Setting}\label{S:mixingness}

The results in the preceding sections become more striking and intuitive if we assume that $L^{2}(\widetilde{M})\ominus L^{2}(M)$ is a mixing $M$-$M$ bimodule. This is the case, for example, when $\alpha$ is the s-malleable deformation described in \Cref{S:1-cohomology} associated to the Gaussian action of a mixing  representation. The reason for considering the mixing case is two-fold, with both aspects going back to Popa's key insights in \cite{PopaStrongRigidity}. First,  mixingness automatically grants $Q' \cap \widetilde M \subseteq M$ to any diffuse subalgebra $Q\leq M$.
%(Lemma \ref{L:usemix} below, cf. \cite[Section 3]{PopaStrongRigidity}).
Second, mixingness allows one to ``boost" certain properties from a diffuse subalgebra $Q \leq M$ to various weak normalizers thereof (\cite[Theorem 4.1, Theorem 4.4]{PopaStrongRigidity}, see also \cite[Section 6]{Va06a} and \cite[Theorem 4.5]{PetersonL2}). 

When the property in question is $\alpha$-rigidity, Peterson \cite[Theorem 4.5]{PetersonL2} showed that in the mixing context one can upgrade rigidity of a diffuse subalgebra to the von Neumann algebra generated by its usual normalizer; however the case of the weaker normalizers introduced by Popa has remained out of reach, one key challenge being their lack of multiplicative group structure. The main result of this section bridges that gap, showing that $\alpha$-rigidity of diffuse $Q \le M$ does indeed imply $\alpha$-rigidity of those weak normalizers of $Q$ (see \Cref{C:upgrading rigidity to all the normalizers} and the discussion directly thereafter).

For simplicity, in this section we will also restrict our attention to the tracial case. We begin by recalling Popa's powerful Intertwining-by-Bimodules Theorem, developed in \cite[Theorem A.1]{PopaL2Betti}, \cite[Lemmas 4 and 5]{PopaFG}, and \cite[Section 2]{PopaStrongRigidity}, and fundamental to deformation/rigidity theory.

\begin{thm}[Popa's Intertwining-by-Bimodules Theorem, {\cite[Section 2]{PopaStrongRigidity}}] \label{T:Popa IBBT}
Let $(M, \tau)$ be a tracial von Neumann algebra, $p, q \in \mathcal{P}(M)$ projections, and $P\leq pMp$, $Q\leq qMq$. Then the following are equivalent:
\begin{enumerate}
\item There exist nonzero projections $p_0 \in P$, $q_0 \in Q$, and a normal unital $*$-homomorphism $\theta: p_0Pp_0 \to q_0Qq_0$, together with a nonzero partial isometry $v\in q_0Mp_0$ such that $\theta(x)v = vx$ for all $x \in p_0Pp_0$.
\item There is no net $(u_n)_{n \in I}$ in $\mathcal{U}(P)$ with $\|\E_Q(qxu_nyq)\|_2 \to 0 \text{ for all } x, y \in M$.\label{I:total mixing over Q interwine}
\item There is a subgroup $G\leq \mathcal{U}(P)$ with $P = G''$ which contains no net $(u_n)_{n \in I}$ satisfying $\|\E_Q(qxu_nyq)\|_2 \to 0 \text{ for all } x, y \in M$. \label{I:mixing over Q}
\item There exists a $ P$-$Q$-sub-bimodule $K$ of $pL^2(M)q $ satisfying $\operatorname{dim}(K_Q)<\infty $.\label{I:FD bimodule}
\end{enumerate}
\end{thm}

If the equivalent conditions of the theorem above are met, we say that \emph{a corner of $P$ intertwines into $Q$ inside $M$}, and write $P \prec_M Q$. If $Pp' \prec_M Q$ for every $p' \in P' \cap pMp$, we write $P \prec_M^s Q$.
We direct the reader to \cite[Section 3]{Va07} for a number of basic stability properties of $\prec_M$ and $\prec_M^s$ (see also \cite[Lemma 2.4]{DHI16}), and to \cite[Appendix F]{BO08} for a detailed exposition of the theory.

One of Popa's essential insights in his development of the intertwining techniques above is their natural interplay with the notion of mixingness relative to a subalgebra (\cite[Definition 2.9]{PopaCoc}, \cite[Definition 2.3]{PetersonSinclair}):
\begin{defn} Let $(M, \tau)$ be a tracial von Neumann algebra and $N\leq M$. Let $p\in \mathcal{P}(M)$ and $Q\leq pMp.$ A $Q$-$Q$ bimodule ${}_Q\mathcal{H}_Q$ is \emph{mixing relative to $N$} if any net $(x_n)_{n \in I}$ in $(Q)_1$ with $\|\E_N(yx_nz)\|_2 \to 0$ for all $y, z \in M$ satisfies
\begin{align*}
    \lim_n \sup_{y \in (Q)_1} |\langle x_n\xi y, \eta\rangle| = \lim_n \sup_{y \in (Q)_1} |\langle y\xi x_n, \eta\rangle| = 0 \quad\text{ for all }\quad \xi, \eta \in \mathcal{H}.
\end{align*}
An $M$-$M$ bimodule ${}_M\mathcal{H}_M$ which is mixing relative to $\mathbb{C}$ is simply called \emph{mixing}.
\end{defn}

Popa demonstrated in \cite[Section 3]{PopaStrongRigidity} how the interplay between these two notions leads to control of relative commutants in the presence of mixingness, in our context automatically granting $N' \cap \widetilde M \subseteq M$ to any diffuse subalgebra $N\leq M$. In fact, recalling that $N$ is diffuse if and only if $N \nprec_M \mathbb{C}$, we have the following more general fact, whose short and well known proof we include for the sake of completeness:

\begin{lem}[Cf. {\cite[Section 3]{PopaStrongRigidity}}]\label{L:usemix}
Suppose tracial von Neumann algebras $A < M < \widetilde M$ are such that $L^{2}(\widetilde{M})\ominus L^{2}(M)$ is mixing relative to $A$.

Then for any $q\in \mathcal{P}(M)$ and $Q\leq qMq$, either
\begin{enumerate}
    \item $Q \prec_M A$, or
    \item $Q' \cap q\widetilde{M}q \subseteq qMq$.
\end{enumerate}
\end{lem}
\begin{proof}
Suppose that $Q\nprec_{M}A.$ Take $x \in Q' \cap q\widetilde{M}q$ and set $\delta = x - \E_M(x)$, noting that $\delta \in Q' \cap q\widetilde{M}q$ as well. Since $Q \nprec_M A$, there is a net $(u_n)_{n \in I}$ in $\mathcal{U}(Q)$ such that $\|\E_A(au_nb)\|_2 \to 0$ for all $a, b \in M$. Then
\begin{align*}
2\|\delta\|_2^2 = \|[\delta, u_n]\|_2^2+2\text{Re}\langle\delta u_n, u_n \delta\rangle
= 2\text{Re}\langle\delta u_n, u_n \delta\rangle \to 0,
\end{align*}
because $L^2(\widetilde M) \ominus L^2(M)$ is mixing relative to $A$. Hence $\delta = 0$ and $x \in M$ as desired.
\end{proof}

We will think of the condition $Q\nprec_M N$ as ``$Q$ is diffuse over $N$". We remark that it may be easier to the reader on first  reading to think of the case $N=\C 1$ throughout this section. In this case every instance of ``mixing over $N$" simply becomes ``mixing", and every instance of $Q\nprec_M N$ becomes ``$Q$ is diffuse."

Because of Lemma \ref{L:usemix}, the relative commutant conditions that occur in the statement of Corollary \ref{C:existence of maximal subalgebras} follow from assuming that the appropriate subalgebras are ``diffuse over $N.$" Because of this, we can significantly simplify the statement of Corollary \ref{C:existence of maximal subalgebras} in the mixing situation.

\begin{cor}\label{C:existence of maximal subalgebras2}
Let $(M,\tau)$ be a tracial von Neumann subalgebra, and $N\leq M.$ Suppose we have an s-malleable deformation $\alpha_{t}\colon \widetilde{M}\to \widetilde{M}$ of $M$ such that $L^{2}(\widetilde{M})\ominus L^{2}(M)$ is mixing over $N.$ Fix a projection $p\in M.$ Suppose $Q\leq pMp,$ that $Q\nprec_M N,$ and that $Q$ is $\alpha$-rigid.  Then there is a unique rigid envelope $P$ of $Q$ in $pMp.$ Further, $P$ has the  following properties:
\begin{enumerate}[(i)]
\item If $B\leq pMp$ has $B\cap P\nprec_M N$,  and $B$ is $\alpha$-rigid, then $B\subseteq P.$ \label{I:diffuse absorption2}
\item Suppose that $\sigma\in \Aut(\widetilde{M}),$ and that $\sigma(pMp)=pMp,$ and $\|(\alpha_{t}\circ \sigma-\sigma\circ \alpha_{t})\big|_{P}\|_{\infty,2}\to_{t\to 0}0.$ If $\sigma(P)\cap P\nprec_M N$, then $\sigma(P)=P.$ \label{I:automorphic absorption2}
\end{enumerate}
\end{cor}

As mentioned above, another particularly nice aspect of the mixing situation is that rigidity automatically passes to normalizers, provided that the subalgebra does not interwine into $N.$ That mixingness often allows one to upgrade certain properties to (weak) normalizers is a core idea of Popa \cite[Theorem 4.1, Theorem 4.4]{PopaStrongRigidity}. See also \cite[Section 6]{Va06a} and \cite[Theorem 4.5]{PetersonL2}). 

We give several applications in the mixing context that upgrade rigidity to von Neumann algebras generated by ``weak versions" of the normalizer and show that a maximal rigid subalgebra is automatically \emph{strongly malnormal} in the sense of \cite{PopaWeakInter}. We will in fact show a more general statement, namely that for an $s$-malleable $(\widetilde{M},\alpha,\beta)$ of $(M,\tau),$ for a maximal rigid $P\leq M,$ the $P$-$P$ bimodule structure of $L^{2}(M)\ominus L^{2}(P)$ can be expressed in terms of the $P$-$P$ bimodule structure of $L^{2}(\widetilde{M})\ominus L^{2}(M).$  

Particular versions of ``weak normalizers" we will be concerned with are the one-sided quasi-normalizer, and the wq-normalizer introduced in the introduction, but there are more general ``weak normalizers" where we can also upgrade rigidity to the von Neumann algebras they generate. 
\begin{defn}[cf.\ \cite{PopaStrongRigidity,PopaCohomologyOE,IPP,GalatanPopa}]
Let $(M,\tau)$ be a tracial von Neumann algebra, and $N,Q\leq M$ with $Q\nprec_M N.$ The \emph{wq-normalizer of $Q$ inside of $M$ relative to $N$} is defined to be
\[\mathcal{N}_{M}^{wq}(Q|N)=\{u\in \mathcal{U}(M):uQu^{*}\cap Q\nprec_M N\}.\]
Observe that when $N=\C1,$ this is the usual wq-normalizer $\mathcal{N}_{M}^{wq}(Q)$.
\end{defn}
We also recall the \emph{weak intertwining space} of Popa.
Suppose that $(M,\tau)$ is a tracial von Neumann algebra, and that $B,Q\leq M.$ 
As in \cite{PopaStrongRigidity, PopaInterwineSpace, PopaWeakInter}, if $\xi\in L^{2}(M),$ we let
\[L^{2}(B\xi Q)=\overline{\Span\left\{b\xi a:b\in B,a\in Q\right\}}^{\|\cdot\|_{2}}.\]
In \cite{PopaInterwineSpace, PopaWeakInter}, Popa defined the intertwining space, denoted $I_{M}(B,Q),$ to be the set of $x\in M$ so that $L^{2}(BxQ)$ is finite-dimensional over $Q$ (this was also discussed in \cite{PopaMC}).
In \cite[Section 2.6]{PopaWeakInter}, Popa also defined the \emph{weak intertwining space}, denoted $wI_{M}(B,Q),$ to be
\[wI_{M}(B,Q)=\bigcup_{B_{0}\leq B\textnormal{ diffuse}}I_{M}(B_{0},Q).\]

Naturally, for $N \le M$, $p\in \mathcal{P}(M),$ and $B,Q\leq pMp,$ the space of weak intertwiners in $pMp$ relative to $N$ is given by
\[wI_{pMp}(B,Q \,|\, N)= \bigcup_{\stackrel{B_0 \le B}{B_0 \nprec_M N}} I_{pMp}(B_0, Q).
\]

Our approach to to upgrade rigidity to the von Neumann algebra generated by these weak normalizers will be by directly exploiting the bimodule structure.
Recall that two $Q$-$Q$ bimodules $\mathcal{H},\mathcal{K}$ are \emph{disjoint} if every bounded, $Q$-$Q$ bimodular map $T\colon \mathcal{H}\to \mathcal{K}$ is zero. Given $Q$-$Q$ bimodules $\mathcal{H},\mathcal{K}$
let $\mathcal{H}_{s}$ be the set $\xi\in \mathcal{H}$ so that for every bounded $Q$-$Q$ bimodular $T\colon \mathcal{H}\to \mathcal{K}$ we have that $T(\xi)=0.$ By construction,
\[\mathcal{H}_{s}=\bigcap_{ T\in B(\mathcal{H},\mathcal{K})\textnormal{ $Q$-$Q$ bimodular}}\ker(T).\]
So $\mathcal{H}_{s}$ is a closed $Q$-$Q$ subbimodule of $\mathcal{H}.$ We call $\mathcal{H}_{s}$ the space of \emph{$\mathcal{K}$-singular vectors}. If $\mathcal{K}=L^{2}(Q)\otimes L^{2}(Q)$ 
and $\mathcal{H}=L^{2}(M)$ for a trace-preserving inclusion $Q\leq M,$ then following \cite{Hayes8} we call $\mathcal{H}_{s}$ the \emph{singular subspace of $Q\subseteq M$}, and denote it by $\mathcal{H}_{s}(Q\subseteq M).$ It is a folklore result (cf. \cite[Proposition 3.3]{Hayes8}) that if we set $\mathcal{H}_{a}=\mathcal{H}\cap (\mathcal{H}_{s})^{\perp},$ then $\mathcal{H}_{a}$
isometrically embeds into an infinite direct sum of $\mathcal{K}$ as a $Q$-$Q$ bimodule.

For a tracial von Neumann algebra and $X\subseteq L^{2}(M,\tau),$ we will need to make sense of $W^{*}(X).$ Recall that a closed, densely-defined, unbounded operator $T$ on $L^{2}(M,\tau)$ is \emph{affiliated to $M$} if its graph, regarded as a closed subspace of $L^{2}(M,\tau)\oplus L^{2}(M,\tau),$ is invariant under the diagonal right action of $M$ on $L^{2}(M,\tau)\oplus L^{2}(M,\tau).$ As is well known (see \cite[Proposition 7.2.3]{APBook}) if $T$ is a closed, densely-defined unbounded operator on $L^{2}(M,\tau),$ and $T=U_{T}|T|$ is its polar decomposition, then $T$ is affiliated to $M$ if and only if $U_{T}\in M,$ and $1_{E}(|T|)\in M$ for every Borel $E\subseteq [0,\infty).$ If $X$ is a collection of closed, densely defined operators affiliated to $M$, we let $W^{*}(X)$ be the von Neumann algebra generated $\{U_{T}:T\in X\}\cup \{1_{E}(|T|):T\in X,\mbox{  $E\subseteq \C$ Borel}\}.$  Recall \cite[Theorem 7.3.2]{APBook} that every $\xi\in L^{2}(M,\tau)$ gives rise to a closed, densely-defined operator $L_{\xi}$ on $L^{2}(M,\tau)$ as follows. Let $T_{\xi}$ be the densely-defined operator on $L^{2}(M,\tau)$ whose domain is $M$ and which is defined by $T_{\xi}(x)=\xi x$ for all $x\in M.$ We then let $L_{\xi}$ be the closure of $T_{\xi}.$ In particular, we can make sense of $W^{*}(X)$ for $X\subseteq \mathcal{H}.$

We now show that one can upgrade rigidity of a subalgebra to the von Neumann algebra generated by the space of vectors singular with respect to the orthocomplement bimodule. The proof follows quickly from \Cref{P:two sided estimate intertwiner}.

\begin{thm}\label{T:upgrading rig to the singualr part}
 Let $(M,\tau)$ be a tracial von Neumann algebra, and let $(\widetilde{M},\alpha,\beta)$ be an $s$-malleable deformation of $(M,\tau).$ Suppose that $p$ is a projection in $M$ and that $Q\leq pMp$ is $\alpha$-rigid. 
 
View $L^{2}(p\widetilde{M}p)\ominus L^{2}(pMp),$ $L^{2}(pMp)$ as $Q$-$Q$ bimodules, and let $\mathcal{H}_{s}$ be the space of $L^{2}(p\widetilde{M}p)\ominus L^{2}(pMp)$-singular vectors in $L^{2}(M).$ Then $W^{*}(\mathcal{H}_{s})$ is $\alpha$-rigid.
\end{thm}

\begin{proof}
Fix $t \in \R$ and adopt the notations of \Cref{P:two sided estimate intertwiner}. We claim that $\gamma_t(W^{*}(\mathcal{H}_{s})) \le \delta_t(Q)$. To establish this, taking any $v \in \mathcal{V}_t$ with $vx = \alpha_t(x)v$ for all $x \in Q$, we will show that in fact $v^{*}\alpha_{t}(W^{*}(\mathcal{H}_{s}))v\subseteq M$.

For ease of notation, define $\Theta_{t}\in \Aut(p\widetilde{M}p,\tau)$ by $\Theta_{t}(x)=v^{*}\alpha_{t}(x)v.$ Observe that $\Theta_{t}\big|_{Q}=\id\big|_{Q}$ for all $t\in \R.$ Since $\Theta_{t}$ is a trace-preserving $*$-automorphism $p\widetilde{M}p\to p\widetilde{M}p,$ we can extend it to a unitary on $L^{2}(p\widetilde{M}p)$ which we will still denote by $\Theta_{t}.$ By similar remarks, we may regard $\E_{pMp}$ as a projection operator on $L^{2}(p\widetilde{M}p).$ We may thus define an operator \[T_{t}\in B(L^{2}(pMp),L^{2}(p\widetilde{M}p)\ominus L^{2}(pMp))\] by $T_{t}=(1-\E_{pMp})\circ \Theta_{t}.$ By the $pMp$-$pMp$ bimodularity of $\E_{pMp},$ and the fact that $\Theta_{t}\big|_{Q}=\id\big|_{Q},$ we know that $T_{t}$ is $Q$-$Q$ bimodular. 
By definition of $\mathcal{H}_{s}$ we thus see that $T_{t}|_{\mathcal{H}_{s}}=0,$ which implies that $\Theta_{t}(\mathcal{H}_{s})\subseteq L^{2}(pMp).$ Hence $\Theta_{t}(W^{*}(\mathcal{H}_{s}))\subseteq M$ as desired, establishing $\gamma_t(W^{*}(\mathcal{H}_{s})) \le \delta_t(Q)$. Then \Cref{P:two sided estimate intertwiner} gives
\begin{align*}
\varepsilon_{2t}(W^{*}(\mathcal{H}_{s})) \le 4\gamma_t(W^{*}(\mathcal{H}_{s})) \le 4\delta_t(Q) \le 24\varepsilon_t(Q),
\end{align*}
so the $\alpha$-rigidity of $Q$ implies that of $W^{*}(\mathcal{H}_{s})$. \qedhere

\end{proof}

 We now give several corollaries on the structure of maximal rigid algebras in the mixing case.
Recall (see \cite{PopaWeakInter}) that Popa defined $Q \le M$ to be \emph{strongly malnormal in $M$} if $wI_{M}(Q,Q)\subseteq Q$.
Of course, if $N \le M$ and $p\in \mathcal{P}(M)$, then $Q\leq pMp$ is \emph{strongly malnormal in $M$ relative to $N$} if $wI_{pMp}(Q,Q\,|\,N)\subseteq Q$. In \cite{PopaStrongRigidity} (see also \cite{PopaWeakInter}), Popa established a precise and fundamental connection between mixing properties of the inclusion $Q\leq M$ and malnormality of $Q\leq M,$ in the sense of either being strongly malnormal or containing the one-sided quasi normalizer. Exploiting this connection lends us the following corollary. For the proof, we recall that if $(M,\tau)$ is a tracial von Neumann algebra, and $x\in M,$ then we set $\|x\|_{1}=\tau(|x|).$

\begin{cor}\label{C:upgrading rigidity to all the normalizers}
Let $(M,\tau)$ be a tracial von Neumann algebra, and let $(\widetilde{M}.\alpha,\beta)$ be an $s$-malleable deformation of $M.$ 
Suppose that as an $M$-$M$ bimodule, $L^{2}(\widetilde{M})\ominus L^{2}(M)$ is mixing relative to $N \le M$, and let $p$ be a projection in $M$.
\begin{enumerate}[(i)]
\item For every maximal rigid $P\leq pMp$ with we have that $L^{2}(M)\ominus L^{2}(P)$ is mixing relative to $N$ as a $P$-$P$ bimodule. \label{I:max rigid is mixing}
\item If $Q\leq pMp$ is $\alpha$-rigid and $Q\nprec_M N$, then $W^*(wI_{pMp}(Q,Q \,|\, N))$ is $\alpha$-rigid. Hence $W^{*}(\mathcal{N}^{wq}_{pMp}(Q|N))$ and $W^{*}(q^{1}\mathcal{N}_{pMp}(Q))$  are $\alpha$-rigid.\label{I:upgrading to one-sided quasi-normalizer}
\item If $P\leq pMp$ is maximal rigid and $P\nprec_M N,$ then $P$ is strongly malnormal relative to $N.$ \label{I:max rig is strongly malnormal}
\item Suppose that $M\leq \widetilde{M}$ is a coarse inclusion in the sense of Popa \cite{PopaWeakInter}. Then for every maximal rigid $P\leq pMp,$ we have that the inclusion $P\leq pMp$ is coarse. \label{I:max rig is coarse}
\item Suppose that $M\leq \widetilde{M}$ is a coarse inclusion. Then for every diffuse, rigid $Q\leq pMp,$ we have that $W^{*}(\mathcal{H}_{s}(Q\subseteq pMp))$ is $\alpha$-rigid. \label{I:upgrading to the singular subspace}
\end{enumerate}
\end{cor}

\begin{proof}

(\ref{I:max rigid is mixing}): Let $\mathcal{H}_{s}$ be as in the statement of Theorem \ref{T:upgrading rig to the singualr part}. Since $P$ is maximal rigid, we have that $\mathcal{H}_{s}\subseteq L^{2}(P).$ By the remarks preceding Theorem \ref{T:upgrading rig to the singualr part}, we know that  as a $P$-$P$ bimodule
$L^{2}(pMp)\ominus L^{2}(P)$ embeds into an infinite direct sum of $L^{2}(p\widetilde{M}p)\ominus L^{2}(pMp).$
By mixingness of $L^{2}(p\widetilde{M}p)\ominus L^{2}(pMp)$ relative to $N,$ we know that $L^{2}(pMp)\ominus L^{2}(P)$ is a mixing relative to $N$ as a $P$-$P$ bimodule.

(\ref{I:upgrading to one-sided quasi-normalizer}): Letting $P$ be the rigid envelope of $Q$, it is enough to show that for any $Q_0 \le Q$ with $Q_0 \nprec_M N$, we have $I_{pMp}(Q_0, Q) \subseteq P$. 
Fix any such $Q_0$ and any net $(v_{n})_{n}$ in $\mathcal{U}(Q_0)$ with 
\[\lim_{n\to\infty}\|\E_{N}(av_{n}b)\|_{2}=0\mbox{ for all $a,b\in M.$}\]
Then, for all $a,b\in pMp$ with $\E_{P}(a)=0=\E_{P}(b)$ we have, by mixingness of $L^{2}(pMp)\ominus L^{2}(P)$ relative to $N$ as a $P$-$P$ bimodule:
\[0=\lim_{n\to\infty}\sup_{y\in P:\|y\|\leq 1}|\tau(b^{*}v_{n}ay)|=\lim_{n\to\infty}\sup_{y\in P:\|y\|\leq 1}|\tau(\E_{P}(b^{*}v_{n}a)y)|=\lim_{n\to\infty}\|\E_{P}(b^{*}v_{n}a)\|_{1}.\]
So for each $n,$
\[\|\E_{P}(b^{*}v_{n}a)\|_{2}^{2}\leq \|\E_{P}(b^{*}v_{n}a)\|_{1}\|\E_{P}(b^{*}v_{n}a)\|\leq \|\E_{P}(b^{*}v_{n}a)\|_{1}\|a\|\|b\|\to_{n\to\infty}0.\]
Hence $I_{pMp}(Q_0, Q) \subseteq P$ follows by \cite[Theorem 3.1]{PopaStrongRigidity} (see also \cite[Proposition 2.6.3]{PopaWeakInter}, \cite[Lemma D.3]{Va06a}).
The ``hence" part follows, since $q^{1}\mathcal{N}_{pMp}(Q),\mathcal{N}_{pMp}^{wq}(Q|N)\subseteq wI_{pMp}(Q,Q\,|\,N),$ (the inclusion $\mathcal{N}_{pMp}^{wq}(Q|N)\subseteq  wI_{pMp}(Q,Q\,|\,N)$ was noted in \cite{IPP, PopaWeakInter}).

(\ref{I:max rig is strongly malnormal}):  $P\nprec_M N$ implies $P \subseteq W^{*}(wI_{pMp}(P,P\,|\,N))$ which is $
\alpha$-rigid by (\ref{I:upgrading to one-sided quasi-normalizer}). Thus the maximality of $P$ forces $W^{*}(wI_{pMp}(P,P\,|\,N)) \subseteq P$ as desired.

% For the case of $\mathcal{N}_{M}^{wq}(Q|N),$ first observe that $\mathcal{N}_{M}^{wq}(Q|N)\subseteq \mathcal{N}_{M}^{wq}(P|N).$ One can now use that by \cite{IPP,GalatanPopa},  $P$ being strongly malnormal implies that $\mathcal{N}_{M}^{wq}(P|N)\subseteq P.$ See also e.g., \cite[Proposition 3.2]{Hayes8}), for another proof that the elements of $\mathcal{N}_{M}^{wq}(Q|N),$ when viewed inside $L^{2}(M),$ are all singular vectors with respect to any $Q$-$Q$ bimodule which is mixing relative to $N.$

(\ref{I:max rig is coarse}):
 By assumption $L^{2}(\widetilde{M})\ominus L^{2}(M)$ embeds into an infinite direct sum of the coarse $M$-$M$ bimodule. So as $pMp$-$pMp$ bimodules we have that $L^{2}(p\widetilde{M}p)\ominus L^{2}(pMp)$ embeds into an infinite direct sum $L^{2}(pMp)\otimes L^{2}(pMp).$ Since $P\leq pMp,$ if we view $L^{2}(pMp)\otimes L^{2}(pMp)$ as a $P$-$P$ bimodule, then it embeds into an infinite direct sum of the coarse $P$-$P$ bimodule. We may now argue as in part (\ref{I:max rigid is mixing}) to complete the proof.

(\ref{I:upgrading to the singular subspace}): Let $P$ be the rigid envelope of $Q.$ Then
\[\mathcal{H}_{s}(Q\subseteq pMp)\subseteq \mathcal{H}_{s}(P\subseteq pMp)=L^{2}(P),\]
the last equality following by part (\ref{I:max rig is coarse}).
Hence $W^{*}(\mathcal{H}_{s}(Q\subseteq pMp))\subseteq P,$ and thus $W^{*}(\mathcal{H}_{s}(Q\subseteq pMp))$ is $\alpha$-rigid.\qedhere
\end{proof}

We remark that it is direct to establish, 
% along the lines of the proof of Lemma \ref{L:normalizer convergence},
as in \cite[Theorem 4.5]{PetersonL2}, 
that if $Q$ is $\alpha$-rigid, and $Q\nprec_M N,$ then $\alpha_{t}$ converges uniformly to the identity as $t\to 0$ on $\mathcal{N}_{M}^{wq}(Q|N)$ in $\|\cdot\|_{2}.$ However, this is not enough to imply that $W^{*}(\mathcal{N}_{M}^{wq}(Q|N))$ is $\alpha$-rigid, because \emph{$\mathcal{N}_{M}^{wq}(Q|N)$ is not a group}. So we cannot apply the same convex hull arguments as in \cite[Proposition 5.1]{PopaL2Betti},\cite[Theorem 4.5]{PetersonL2} 
% (like we did in \Cref{L:normalizer convergence}) 
to upgrade uniform convergence on $\mathcal{N}_{M}^{wq}(Q|N)$ to uniform convergence on the unit ball of $W^{*}(\mathcal{N}_{M}^{wq}(Q|N)).$ In fact, the only proof that we know of rigidity of $W^{*}(\mathcal{N}_{M}^{wq}(Q|N))$ uses either the existence of rigid envelopes, or directly uses Theorem \ref{T:upgrading rig to the singualr part}.

Recall the following notion due to Popa \cite[Definition 2.3]{PopaCohomologyOE}: if $H$ is a subgroup of a discrete group $G,$ then $H$ is \emph{$wq$-normal} in $G$ if there is a chain $(G_{\alpha})_{\alpha}$ indexed by ordinals at most  some given ordinal $\gamma,$ such that $G_{\gamma}=G,G_{0}=H,$ and
\begin{itemize}
    \item $G_{\alpha}=\mathcal{N}_{G}^{wq}(G_{\alpha-1})$ if $\alpha$ is a successor ordinal, and
    \item  $G_{\alpha}=\bigcup_{\beta<\alpha}G_{\beta}$ if $\alpha$ is a limit ordinal.
\end{itemize}
Combining Corollary \ref{T:PropT intro} with Corollary \ref{C:upgrading rigidity to all the normalizers}, we obtain that von Neumann algebras of groups which have wq-normal, infinite, property (T) subgroups are rigid with respect to every s-malleable deformation $(\widetilde{M},\alpha,\beta)$ with $L^{2}(\widetilde{M})\ominus L^{2}(M)$ a mixing $M$-$M$ bimodule. Compare with \cite[Corollary 0.2]{PopaCoc} and \cite[Lemma 2.4]{PopaCohomologyOE}.

As an application of Corollary \ref{C:upgrading rigidity to all the normalizers}, we can also say in the mixing situation that if $P,Q$ are maximal rigid subalgebras which are diffuse then no subalgebra of $P$ which is diffuse can intertwine into $Q$ unless $P,Q$ have unitarily conjugate corners.

\begin{thm}\label{T:generalized dichotomy}
Let $(M,\tau)$ be a tracial von Neumann algebra, and $\alpha_{t}\colon \widetilde{M}\to \widetilde{M}$ an s-malleable deformation of $M.$ Suppose that $L^{2}(\widetilde{M})\ominus L^{2}(M)$ is a mixing  as an $M$-$M$ bimodule. Let $p,q$ be projections in $M$ and $Q_{0}\leq qMq$ be $\alpha$-rigid, and $P\leq pMp$ maximal rigid. Suppose that $Q_{0},P$ are diffuse. Let $Q\leq qMq$ be the rigid envelope of $Q_{0}.$ If $Q_{0}\prec_M P,$ then there are nonzero projections $e\in P,f\in Q$ and a $u\in \mathcal{U}(M)$ so that $u(fQf)u^{*}=ePe.$

\end{thm}

\begin{proof}

The fact that $Q_{0}\prec_M P$ means that there are projections $e_{0}\in P,f_{0}\in Q_{0},$ a unital, normal $\ast$-homomorphism $\theta\colon f_{0}Q_{0}f_{0}\to e_{0}Pe_{0},$ and a nonzero partial isometry $v\in M$ so that:
\begin{itemize}
    \item $xv=v\theta(x)$ for all $x\in f_{0}Q_{0}f_{0},$
    \item $vv^{*}\in (f_{0}Q_{0}f_{0})'\cap f_{0}Mf_{0},$
    \item $v^{*}v\in \theta(f_{0}Q_{0}f_{0})'\cap e_{0}Me_{0}.$
\end{itemize}

Set $f=vv^{*},e=v^{*}v.$ We first claim that $e\in P,f\in Q.$ Let us begin by showing that $f\in Q.$ To prove this, set
\[\widehat{Q}_{0}=f_{0}Q_{0}f_{0}\vee [(f_{0}Q_{0}f_{0})'\cap f_{0}Mf_{0}]\]
Then $\widehat{Q}_{0}\leq W^{*}(N_{f_{0}Mf_{0}}(f_{0}Q_{0}f_{0})),$ so by Corollary \ref{C:upgrading rigidity to all the normalizers} we know that $\widehat{Q}_{0}$ is $\alpha$-rigid. By Corollary \ref{C:corners of pineapples} (\ref{I:compressing down}), $f_{0}Qf_{0}$ is maximal rigid. Since $Q$ is diffuse and $_{M}[L^{2}(\widetilde{M})\ominus L^{2}(M)]_{M}$ is mixing, we know that $\widehat{Q}_{0}$ has a unique rigid envelope. Since $\widehat{Q}_{0}\supseteq f_{0}Q_{0}f_{0},$ the rigid envelope of $\widehat{Q}_{0}$ is the same as the rigid envelope of $f_{0}Q_{0}f_{0}.$  Thus we must have that $\widehat{Q}_{0}\subseteq f_{0}Qf_{0}.$ But $f\in \widehat{Q}_{0}$ by definition, so $f\in Q.$

To see that $e\in P,$ let
\[\widehat{P}_{0}=(\theta(f_{0}Q_{0}f_{0})'\cap e_{0}Me_{0})\vee \theta(f_{0}Q_{0}f_{0}).\]
We first claim that $\widehat{P}_{0}$ is $\alpha$-rigid. Since $\theta(f_{0}Q_{0}f_{0})$ is the image of a diffuse algebra under a nonzero normal homomorphism, we know that $\theta(f_{0}Qf_{0})$ is diffuse. So we can argue as in the first paragraph to see that $\widehat{P}_{0}$ is rigid. But then we can argue again as in the first paragraph to see that $\widehat{P}_{0}\subseteq e_{0}Pe_{0},$ and thus $e\in P.$

Since $v^{*}v=f\in Q,vv^{*}=e\in P,$ it is easy to see that $v^{*}(fQf)v$ is a subalgebra of $eMe,$ and it is easily seen to be rigid. Additionally,
\[v^{*}(fQf)v\cap ePe\geq e\theta(f_{0}Q_{0}f_{0}).\]
As $e\theta(f_{0}Qf_{0})$ is diffuse, we have that $v^{*}(fQf)v\cap ePe$ is diffuse. Since $_{M}[L^{2}(\widetilde{M})\ominus L^{2}(M)]_{M}$ is mixing, we know that $v^{*}(fQf)v$ is contained in a unique rigid envelope. Since $v^{*}(fQf)v\cap ePe$ is diffuse and $ePe$ is maximal rigid by Corollary \ref{C:corners of pineapples} (\ref{I:compressing down}), it follows that the rigid envelope of $v^{*}(fQf)v$ is $ePe.$ So $v^{*}(fQf)v\subseteq ePe.$ But then $fQf\subseteq v(ePe)v^{*},$ and since $v^{*}v=e\in P$ we know that $v(ePe)v^{*}$ is a rigid subalgebra of $fMf.$ Since $fQf$ is maximal rigid, this implies that $v(ePe)v^{*}=fQf.$ Since $M$ is finite, we may find a $u\in \mathcal{U}(M)$ so that $ue=v,$ and then we have that $u(ePe)u^{*}=fQf.$ \qedhere

\end{proof}

We have the following dichotomy for weak intertwining spaces of maximal rigid subalgebras.

\begin{cor}
Let $(M,\tau)$ be a tracial von Neumann algebra, and $\alpha_{t}\colon \widetilde{M}\to \widetilde{M}$ an s-malleable deformation of $M.$ Suppose that $_{M}[L^{2}(\widetilde{M})\ominus L^{2}(M)]_{M}$ is mixing. Let $p,q\in M$ be projections. If $P\leq pMp,Q\leq qMq$ are two diffuse maximal rigid subalgebras of $M,$ then exactly one of the following two occurs:
\begin{itemize}
    \item either there are nonzero projections $e\in P,f\in Q$ and a unitary $u\in \mathcal{U}(M)$ so that $u(ePe)u^{*}=fQf,$ or
    \item $wI_{M}(Q,P)=\{0\}.$
\end{itemize}

\end{cor}

\begin{proof}
 By Part \ref{I:FD bimodule} of Theorem \ref{T:Popa IBBT}, $wI_{M}(Q,P)\ne \{0\}$ if and only if there is a diffuse $Q_{0}\leq Q$ so that $Q_{0}\prec_M P.$ So Theorem \ref{T:generalized dichotomy} implies that if $wI_{M}(Q,P)\ne \{0\},$ then there are projections $e\in P,f\in Q$ and a unitary $u\in \mathcal{U}(M)$ so that $u(ePe)u^{*}=fQf.$ \qedhere

\end{proof}

% For this, we use the following definition.

% \begin{defn}\label{D:twisted bimodules}
% Let $(M,\tau)$ be a tracial von Neumann algebra which is embedded in another tracial von Neumann algebra $(\widetilde{M},\tau)$ in a trace-preserving way. Fix a projection $p\in M$ and $Q\leq pMp.$ Suppose that $\sigma\in \Aut(p\widetilde{M}p,\tau)$ and that $\sigma(Q)\subseteq pMp.$ We define the \emph{$\sigma$-twisted orthocomplement bimodule} to be the Hilbert space $L^{2}(\widetilde{M})\ominus L^{2}(M)$ with the $Q$-$Q$ bimodule structure:
% \[x\cdot \xi \cdot y=\sigma(x)\xi\sigma(y)\mbox{ for all $x,y\in Q$.}\]
% We use $(L^{2}(\widetilde{M})\ominus L^{2}(M))^{\sigma}$ to denote this $Q$-$Q$ bimodule.
% \end{defn}

\section{Applications to $L^2$-rigidity}

$L^2$-rigidity is a von Neumann algebraic analog of vanishing of first $\ell^2$-Betti number for countable, discrete groups. It is not hard to see that for a countable, discrete group $G$ with positive first $\ell^2$-Betti number, the corresponding group von Neumann algebra $L(G)$ is not $L^2$-rigid. The converse is a well-known open problem in the subject. Nonetheless, there are many well established parallels between $L^2$-rigidity and vanishing of the first $\ell^2$-Betti number. For example, as with $\ell^{2}$-Betti numbers, one can show that $L^{2}$-rigidity of $L(G)$ is invariant under orbit equivalence \cite{PetersonSinclair}. $L^{2}$-rigidity of $L(G)$ implies cocycle superrigidity of Bernoulli shifts, and similarly having positive first $\ell^2$-Betti number is an obstruction to even $\mathbb T$-cocycle superrigidity (see \cite[Section 5]{PetersonSinclair}).

Let $(M,\tau)$ be a tracial von Neumann algebra, and let $\mathcal H$ be an $M$-$M$ bimodule. We assume that there is an antilinear involution $J: \cal H\to \cal H$ such that $J(x\xi y) = y^*J(\xi)x^*$ for all $x,y\in M$ and $\xi\in \mathcal H$, in which case we say that $\mathcal H$ comes equipped with a \emph{ real structure}. Consider an $M$-$M$ bimodule with real structure $(\mathcal H, J)$. Given a closeable, unbounded operator $\delta: L^2(M)\to \mathcal H$, we write $D(\delta)$ for its domain. We say that an unbounded operator $\delta: L^2(M)\to \mathcal H$ is a \emph{closeable, real derivation} if it closeable as an operator, $D(\bar\delta)\cap M$ is an ultraweakly dense $\ast$-subalgebra of $M$ on which $\delta$ acts as a derivation, i.e., \[\bar\delta(xy) = x\bar\delta(y) + \bar\delta(x)y\ \textup{for all}\ x,y\in D(\bar\delta)\cap M,\] and $J\bar\delta(x) = \bar\delta(x^*)$ for all $x\in D(\bar\delta)\cap M$.

For a proof of the following proposition, see \cite[section 2]{PetersonL2}.
\begin{prop} To every derivation there is an associated one-parameter Markov semigroup of trace-preserving u.c.p.\ maps $\vp_t: M\to M$ given by $$\vp_t(x) := \exp(-t\delta^*\overline{\delta}(x))$$ for all $x\in D(\bar\delta)\cap M$. The semigroup is continuous in the topology of pointwise strong convergence.
\end{prop}

The following definition essentially appears as Definition 2.13 in \cite{PetersonSinclair} and is a technical modification of \cite[Definition 4.1]{PetersonL2}.

\begin{defn}\label{L2rigid} Let $(M,\tau)$ be a tracial von Neumann algebra, $p\in \mathcal P(M)$, and $Q\leq pMp$. We say that $Q$ is \emph{$L^2$-rigid} if for every tracial inclusion $(M\subseteq N,\tau_N)$ into any tracial von Neumann algebra $N$ and any closeable, real derivation $\delta: L^2(N)\to \cal H$ with ${}_M\cal H_M$ embeddable into a countable direct sum of coarse $M$-$M$-bimodules, for the associated Markov semigroup $(\vp_t)$, $\|(\vp_t-\id)|_Q\|_{\infty,2}\to 0$ as $t\to 0.$

\end{defn}

The Markov semigroup gives a deformation by u.c.p.\ maps, which is suitable in many instances for running a deformation/rigidity argument. However, this deformation was observed to act poorly with respect to algebraic constructs, such as for the case of tensor products of bimodules \cite{SinclairGaussian, ChifanSinclair}, and in those cases it seems necessary with current techniques that the Markov semigroup ``lifts'' in the sense described below to an s-malleable deformation. With this in mind we fix the following terminology.

\begin{defn} Let $(M,\tau_M)$ be a tracial von Neumann algebra, and let $(\vp_t)_{t\geq0}$ be a pointwise-strongly continuous one-parameter semigroup of trace-preserving u.c.p.\ maps. We say that $(\vp_t)$ admits an \emph{s-malleable dilation} $(\widetilde M, \alpha, \beta)$, if $M\subseteq \widetilde M$ is a tracial inclusion of finite von Neumann algebras and $(\alpha_t,\beta)$ is an s-malleable deformation of $\widetilde M$ such that $\vp_{f(t)}(x) =\E_M(\alpha_t(x))$ for all $x\in M$ and $t\in\mathbb R$ for some continuous, open map $f:\mathbb R\to \mathbb R_+$ so that $f(0) = 0$.
\end{defn}

In the case of a group von Neumann algebra $L(G)$ and a derivation arising from a $1$-cocycle: $b: G\to \cal H$, we have that the Markov semigroup is described by the one-parameter semigroup of Schur multipliers $\vp_t(g) := \exp(-t\|b(g)\|^2)$ on $L(G)$, and we have that the $s$-malleable deformation described in Definition \ref{D:s-mall def cocycle}
gives an s-malleable dilation of this Markov semigroup with $f(t)=t^2$.  Working with general derivations and their Markov semigroups is much more challenging, but there is a natural framework for dilation theory in the theory of free stochastic differential equations, \cite{Dabrowski, JRS}. The following theorem is due to Y.\ Dabrowski and appears as a special case of \cite[Theorem 20]{Dabrowski} combined with \cite[Proposition 31]{Dabrowski}.
%See also \cite[Theorem 3.1]{DabrowskiIoana}.

\begin{thm}[Dabrowski]\label{yoann} Let $M$ be a II$_1$ factor and let $\delta: M\to [L^2(M)\otimes L^2(M)]^{\oplus \infty}$ be a closeable, real derivation. Then $\exp(-t\delta^{*}\overline{\delta})$ admits an s-malleable dilation $(\widetilde M, \alpha,\beta)$ so that $L^2(\bigvee_{t\in\mathbb [0,\infty)} \alpha_t(M))\ominus L^2(M)$ is embeddable as an $M$-$M$ bimodule in a countable direct sum of the coarse $M$-$M$ bimodule.
\end{thm}

\noindent Note that the dilation of the Schur multipliers $\vp_t$ above obtained from Dabrowski's techniques, which scales for $f(t) = |t|$, is not the same dilation which is obtained as in Definition \ref{D:s-mall def cocycle}.

\begin{defn}\label{defn:approxL2rigid}
Let $(M,\tau)$ be a tracial von Neumann algebra and $\mathcal D = (\widetilde{M}_i,\alpha_i,\beta_i)_{i\in I}$ be a family of s-malleable deformations of $M.$ Let $p\in \mathcal P(M)$ and $Q\leq pMp$. We say that $Q$ is:
\begin{enumerate}[(i)]
    \item \emph{$\mathcal D$-rigid} if $Q$ is $\alpha_i$-rigid for all $i\in I$.
    \item \emph{approximately $\mathcal D$-rigid} if there is a (potentially nonunital) increasing sequence $(Q_{n})_{n}$ so that $Q_{n}\leq p_nQp_n,$ each $Q_{n}$ is $\mathcal D$-rigid, and $Q=\bigvee_{n}Q_{n}.$ We call such a sequence $(Q_{n})_{n}$ a \emph{$\mathcal D$-rigid filtration of $Q$}.
    \item \emph{approximately $L^{2}$-rigid} if there is an increasing sequence $(Q_n)$ with $Q_n\leq p_nMp_n$ $L^2$-rigid for each $n$ and $Q=\bigvee_n Q_n$. We then refer to $(Q_n)_n$ as an \emph{$L^{2}$-rigid filtration} of $Q$.

\end{enumerate}

\end{defn}

\begin{rmk}\label{rmk:DL2} Given a tracial von Neumann algebra $(M,\tau)$, we denote by $\mathcal D_{L^2}$ the family of all s-malleable dilations $(\widetilde M, \alpha, \beta)$ of $M$ whose orthocomplement bimodule $L^2\left(\bigvee_{t\in [0,\infty)}\alpha_{t}(M)\right)\ominus L^2(M)$ embeds in a countable direct sum of coarse $M$-$M$ bimodules. We then have that $Q\leq pMp$ is approximately $L^2$-rigid if and only if it is approximately $\mathcal D_{L^2}$-rigid. Indeed the infinitesimal generator $\delta$ of $\alpha$ for any s-malleable deformation $(\widetilde M, \alpha, \beta)$ belonging to $\mathcal D_{L^2}$ can be seen to be a closeable $L^2$-derivation $\delta: M\to L^2(\widetilde M)$ of $M$ valued in an $M$-$M$ bimodule embeddable into a countable direct sum of the coarse $M$-$M$ bimodule with real structure given by the Tomita conjugation operator $J$ on $\widetilde M$. The Markov semigroup associated with $\delta$ converges uniformly on a set if and only if $\alpha$ does by \cite[Corollary 5.2]{PetersonSinclair}.  Dabrowski's work, stated as Theorem \ref{yoann} above, provides the highly nontrivial converse.

Fix a $(\widetilde{M},\alpha,\beta)$ in $\mathcal D_{L^2}.$ The fact that $L^2(\bigvee_{t\in\mathbb [0,\infty)} \alpha_t(M))\ominus L^2(M)$ embeds into an infinite direct sum of the coarse bimodule implies in particular that it is a mixing $M$-$M$ bimodule. 
% This is strong enough of an assumption that we may modify the proof of  Theorem \ref{T:preservation on big interesections} to see that if $p\in \mathcal{P}(M)$ and $Q\leq pMp$ is $\alpha$-rigid and diffuse, then the join of two $\alpha$-rigid von Neumann subalgebras with diffuse intersection is $\alpha$-rigid.
This is strong enough of an assumption that we may modify the proof of  Theorem \ref{T:preservation on big interesections} to see that the join of two $\alpha$-rigid von Neumann subalgebras of a corner of $M$ is rigid, provided the algebras have diffuse intersection.
Similarly every diffuse, $\alpha$-rigid von Neumann subalgebra has a rigid envelope, and inductive limits of diffuse, $\alpha$-rigid von Neumann subalgebras are $\alpha$-rigid.

It may appear that in order to obtain the conclusion of Theorem \ref{T:preservation on big interesections}  to any pair of $\alpha$-rigid subalgebras with diffuse intersection, we need to assume that $L^{2}\left(\bigvee_{t\in \R}\alpha_{t}(M)\right)\ominus L^{2}(M)$ is mixing, because as currently written we need to have that \begin{align}Q'\cap p\left(\bigvee_{t\in \R}\alpha_{t}(M)\right)p\subseteq pMp\label{E:commutant}\end{align}
for any $p\in \mathcal{P}(M)$ and any diffuse $Q\leq M$. Assuming that $L^{2}\left(\bigvee_{t\in [0,\infty)}\alpha_{t}(M)\right)\ominus L^{2}(M)$ embeds into an infinite direct sum of the coarse does not, a priori, imply that $L^{2}\left(\bigvee_{t\in \R}\alpha_{t}(M)\right)\ominus L^{2}(M)$ is mixing.

Let us explain briefly why our assumption that $L^{2}\left(\bigvee_{t\in [0,\infty)}\alpha_{t}(M)\right)\ominus L^{2}(M)$ embeds into an infinite direct sum of the coarse is sufficient to obtain the conclusion of Theorem \ref{T:preservation on big interesections} under the assumption that $Q_{1},Q_{2}$ are $\alpha$-rigid with diffuse intersection. The main ingredient in the proof of Theorem \ref{T:preservation on big interesections} is Proposition \ref{L:intertwining on big sets}.  It is easy to see from the  proof of Proposition \ref{L:intertwining on big sets} that the $v$ in the conclusion of Proposition \ref{L:intertwining on big sets} may be required to be in $\bigvee_{s\in [0,\infty)}\alpha_{s}(M)$ if $t>0.$
From here, one can follow the proof of Theorem \ref{T:preservation on big interesections} to see that the conclusion of Theorem \ref{T:preservation on big interesections} holds whenever $p\in \mathcal{P}(M),$ $Q_{j}\leq pMp,j=1,2$ are $\alpha$-rigid, $Q_{1}\cap Q_{2}$ is diffuse, and $L^{2}\left(\bigvee_{t\in [0,\infty)}\alpha_{t}(M)\right)\ominus L^{2}(M)$ is a mixing $M$-$M$ bimodule; indeed since $\|\alpha_t(x) - x\|_2 = \|\alpha_{-t}(x) - x\|_2$ for all $x \in M$, by considering first only $t > 0$, the above remarks show that the containment
\[Q'\cap p\left(\bigvee_{t\in [0, \infty)}\alpha_{t}(M)\right)p\subseteq pMp\]
given by this mixingness (and Lemma \ref{L:usemix}) is sufficient in place of \eqref{E:commutant}.

\end{rmk}

It is easy to see that $Q\leq M$ of the form $Q = A\oplus \left(\bigoplus_n Q_n\right)$ is approximately $L^2$-rigid if $A$ amenable and each $Q_n\leq M$ is $L^2$-rigid. Using Dabrowski's results, we are able to prove by our techniques that possessing such a decomposition completely characterizes approximate $L^2$-rigidity. Clearly we need to only focus on the case that $Q\leq M$ has no amenable direct summand, in which case we have:

\begin{prop}\label{P:approx L2 rigid}
Let $(M,\tau)$ be a tracial von Neumann algebra, $p\in\mathcal P(M)$, and $Q\leq pMp.$ Suppose that $Q$ has no amenable direct summand.

Suppose that $\mathcal D = (\widetilde M_i, \alpha_i,\beta_i)_{i\in I}$ is a family of s-malleable deformations of $M$ such that, for each $i\in I$, $L^{2}\left(\bigvee_{t\in [0,\infty)}\alpha_{i, t}(M)\right)\ominus L^{2}(M)$ is a mixing $M$-$M$ bimodule. Then $Q$ is approximately $\mathcal D$-rigid if and only if it is a direct sum of $\mathcal D$-rigid algebras.

In particular, we have that $Q$ is approximately $L^{2}$-rigid if and only if it is a direct sum of $L^{2}$-rigid algebras.

\end{prop}

\begin{proof}

Suppose that $Q$ is a direct sum of $\mathcal D$-rigid algebras. First, suppose that $Q=\bigoplus_{n=1}^{\infty}Q_{n},$ where each $Q_{n}$ is nonzero and $\mathcal D$-rigid. Let
\[B_{n}=\bigoplus_{j=1}^{n}Q_{j}.\]
Then the $B_{n}$ are $\mathcal D$-rigid and clearly form a $\mathcal D$-rigid filtration of $Q.$ If we cannot write $Q$ as such an infinite direct sum, it must be the case that $Q=\bigoplus_{j=1}^{n}Q_{j}$ where each $Q_{j}$ is $\mathcal D$-rigid. This of course implies that $Q$ itself is $\mathcal D$-rigid, and so $Q_{n}=Q$ provides an $\mathcal{D}$-rigid filtration of $Q.$

In the converse direction, suppose that $Q$ is approximately $\mathcal D$-rigid, and let $(Q_{n})_{n}$ be an $\mathcal D$-rigid filtration of $Q.$ For each $n,$ we may choose a projection $p_{n}\in Z(Q_{n})$ so that $p_nQ_{n}$ has no amenable direct summand and $(p - p_n)Q_{n}$ is amenable. Since $Q_{n}\subseteq Q_{n+1},$ we must have that $p_{n}\leq p_{n+1},$ and that $p_{n}Q_{n}\subseteq p_{m}Q_{m}$ for all $m\geq n.$  Also, since $Q$ has no amenable direct summand, we must have that $p_{n}\to p$ in the strong operator topology. Since $p_{n}Q_{n}$ has no amenable direct summand, it is diffuse. We also have that $p_{n}Q_{m}p_{n}\geq Q_n$ is $\mathcal D$-rigid, whence $(p_{n}Q_{m}p_{n})_{m\geq n}$ is an increasing sequence of diffuse, $\mathcal D$-rigid algebras. From Corollary \ref{cor:defcontrolsubalgebra} (which, following from Theorem \ref{T:preservation on big interesections}, applies under these hypotheses by Remark \ref{rmk:DL2}), it follows that
\[\bigvee_{m\geq n}p_{n}Q_{m}p_{n}\]
is $\mathcal D$-rigid. But this inductive limit is easily seen to be $p_{n}Qp_{n},$ whence $p_{n}Qp_{n}$ is $\mathcal D$-rigid. Let $z_{n}$ be the central support of $p_{n}$ in $Q.$ By Proposition \ref{P:amplifying up from corners} we know that $z_{n}Q$ is $\mathcal D$-rigid. Additionally, we have that the $z_{n}$'s are increasing and tend to $p.$ So
\[Q=z_{1}Q\oplus \bigoplus_{n=1}^{\infty}(z_{n+1}-z_{n})Q,\]
 thus $Q$ is a direct sum of $\mathcal D$-rigid algebras.

 The final part of the proposition now follows from Remark \ref{rmk:DL2}.\qedhere
\end{proof}

%Briefly, we say that $Q\leq M$ is \emph{approximately $L^2$-rigid} if $Q$ is generated as a von Neumann algebra by an increasing union of $L^2$-rigid subalgebras of $M$. (See Definition \ref{defn:approxL2rigid} below for a more precise statement.)
The notion of approximate $L^2$-rigidity captures in the tracial von Neumann algebra setting the vanishing of \emph{reduced} first $\ell^2$-cohomology for discrete groups, which is meant to address the technical issue that amenable tracial von Neumann algebras are not $L^2$-rigid. It is well-known that for nonamenable groups vanishing of reduced $\ell^2$-cohomology is the same as vanishing of the first $\ell^2$-Betti number by Guichardet's theorem \cite[Proposition 2.12.2]{BHV}. (See also \cite[Corollary 2.4]{PetersonThom}.) In this way Proposition \ref{P:approx L2 rigid} may be seen as an analog of Guichardet's theorem for approximate $L^2$-rigidity, showing that it essentially coincides with $L^2$-rigidity for tracial von Neumann algebras with no amenable direct summand.

It follows from the work of Peterson and Thom, \cite[Theorem 2.2]{PetersonThom}, that for any countable, discrete group $G$ and any two subgroups $H_1,H_2< G$ with $|H_1\cap H_2|=\infty$ it holds that $H_1\vee H_2$ has vanishing first reduced $\ell^2$-cohomology if both $H_1$ and $H_2$ do. Motivated by this result, we conjecture the analogous statement should still hold in the II$_1$ factor case, to whit:

\begin{conj}\label{conj:approxL2rigid} Let $M$ be a II$_1$ factor and $Q_1,Q_2\leq M$ such that $Q_i\leq M$ is approximately $L^2$-rigid for $i=1,2$. If $Q_1\cap Q_2$ is diffuse, then $Q_1\vee Q_2\leq M$ is approximately $L^2$-rigid.

\end{conj}

We note that it follows by  \cite[Remark 4.2.4]{PetersonL2} that any approximately $L^2$-rigid subalgebra of $L(\mathbb F_2)$ is amenable. Thus a positive solution to Conjecture \ref{conj:approxL2rigid} would imply a positive solution of the following conjecture of Peterson and Thom (see the remarks after \cite[Proposition 7.7]{PetersonThom}):

\begin{conj}[Peterson and Thom] If $Q_1,Q_2\leq L(\mathbb F_2)$ are amenable and $Q_1\cap Q_2$ is diffuse, then $Q_1\vee Q_2$ is amenable.

\end{conj}

\bibliographystyle{abbrv}
\bibliography{pineapple}

\end{document}